\documentclass[12pt]{amsart}
\usepackage{graphicx}
\usepackage{amsmath,amsthm,amssymb,enumerate}
\usepackage{euscript,mathrsfs}
\usepackage[left=3cm,right=3cm,top=4.5cm,bottom=4.5cm]{geometry}
\usepackage[citecolor=blue,colorlinks=true]{hyperref}
\usepackage{color}

\usepackage{soul}

\catcode`\@=11 \@addtoreset{equation}{section}

\catcode`\@=12

\allowdisplaybreaks

\newcommand{\cdummy}{\cdot}

\newcommand{\tmmathbf}{\mathbf}

\newtheorem{Theorem}{Theorem}[section]
\newtheorem{Proposition}[Theorem]{Proposition}
\newtheorem{Lemma}[Theorem]{Lemma}
\newtheorem{Corollary}[Theorem]{Corollary}

\theoremstyle{definition}
\newtheorem{Definition}[Theorem]{Definition}

\newtheorem{Remark}[Theorem]{Remark}

\newcommand{\bTheorem}[1]{
\begin{Theorem} \label{T#1} }
\newcommand{\eT}{\end{Theorem}}

\newcommand{\bProposition}[1]{
\begin{Proposition} \label{P#1}}
\newcommand{\eP}{\end{Proposition}}

\newcommand{\bLemma}[1]{
\begin{Lemma} \label{L#1} }
\newcommand{\eL}{\end{Lemma}}

\newcommand{\bCorollary}[1]{
\begin{Corollary} \label{C#1} }
\newcommand{\eC}{\end{Corollary}}

\newcommand{\dxt}{\,\mathrm{d}x\,\mathrm{d}t}

\newcommand{\diver}{\mathrm{div}}
\newcommand{\pas}{$\mathbb{P}$-a.s.}

\newcommand{\torN}{Q}
\newcommand{\intTN}[1]{\int_{\torN} #1 \ \dx}
\newcommand{\bRemark}[1]{
\begin{Remark} \label{R#1} }
\newcommand{\eR}{\end{Remark}}

\newcommand{\bDefinition}[1]{
\begin{Definition} \label{D#1} }
\newcommand{\eD}{\end{Definition}}

\newcommand{\Q}{Q}
\newcommand{\tvW}{\tilde{W}}

\newcommand{\dif}{\mathrm{d}}

\newcommand{\mf}{\mathfrak{F}}

\newcommand{\prst}{\mathbb{P}}

\newcommand{\p}{\mathbb{P}}

\newcommand{\mt}{Q}

\newcommand{\bfu}{\mathbf{u}}

\newcommand{\bfg}{\mathbf{g}}

\newcommand{\bfF}{\mathbf{F}}

\newcommand{\bfphi}{\boldsymbol{\varphi}}
\newcommand{\bfvarphi}{\boldsymbol{\varphi}}

\newcommand{\dd}{\mathrm{d}}

\newcommand{\bFormula}[1]{
\begin{equation} \label{#1}}
\newcommand{\eF}{\end{equation}}

\newcommand{\N}{\mathbb{N}}

\newcommand{\Ov}[1]{\overline{#1}}

\newcommand{\aleq}{\lesssim}

\newcommand{\vr}{\varrho}

\newcommand{\tvr}{\tilde \vr}
\newcommand{\tvu}{{\tilde \vu}}

\newcommand{\vu}{\vc{u}}
\newcommand{\vm}{\vc{m}}

\newcommand{\vc}[1]{{\bf #1}}

\newcommand{\Div}{{\rm div}_x}
\newcommand{\Grad}{\nabla_x}

\newcommand{\dx}{\,{\rm d} {x}}

\newcommand{\dt}{\,{\rm d} t }

\newcommand{\intQ}[1]{\int_{{\Q}} #1 \ \dx}

\newcommand{\D}{{\rm d}}

\newcommand{\ep}{\varepsilon}

\newcommand{\R}{\mathbb{R}}

\newcommand{\expe}[1]{ \mathbb{E} \left[ #1 \right] }

\newcommand{\tmop}{\text}
\def\softd{{\leavevmode\setbox1=\hbox{d}%
          \hbox to 1.05\wd1{d\kern-0.4ex{\char039}\hss}}}
\definecolor{Cgrey}{rgb}{0.85,0.85,0.85}
\definecolor{Cblue}{rgb}{0.50,0.85,0.85}
\definecolor{Cred}{rgb}{1,0,0}
\definecolor{fancy}{rgb}{0.10,0.85,0.10}

\newcommand\Cbox[2]{%
    \newbox\contentbox%
    \newbox\bkgdbox%
    \setbox\contentbox\hbox to \hsize{%
        \vtop{
            \kern\columnsep
            \hbox to \hsize{%
                \kern\columnsep%
                \advance\hsize by -2\columnsep%
                \setlength{\textwidth}{\hsize}%
                \vbox{
                    \parskip=\baselineskip
                    \parindent=0bp
                    #2
                }%
                \kern\columnsep%
            }%
            \kern\columnsep%
        }%
    }%
    \setbox\bkgdbox\vbox{
        \color{#1}
        \hrule width  \wd\contentbox %
               height \ht\contentbox %
               depth  \dp\contentbox
        \color{black}
    }%
    \wd\bkgdbox=0bp%
    \vbox{\hbox to \hsize{\box\bkgdbox\box\contentbox}}%
    \vskip\baselineskip%
}

\date{}

\begin{document}


\title[Compressible fluids with random forcing]{On the long time behavior of compressible fluid flows excited by random forcing}

\author{Dominic Breit}
\address[D. Breit]{Department of Mathematics, Heriot-Watt University, Riccarton Edinburgh EH14 4AS, UK}
\email{d.breit@hw.ac.uk}

\author{Eduard Feireisl}
\address[E.Feireisl]{Institute of Mathematics AS CR, \v{Z}itn\'a 25, 115 67 Praha 1, Czech Republic
\and Institute of Mathematics, TU Berlin, Strasse des 17.Juni, Berlin, Germany }
\email{feireisl@math.cas.cz}
\thanks{The research of E.F. leading to these results has received funding from
the Czech Sciences Foundation (GA\v CR), Grant Agreement
18--05974S. The Institute of Mathematics of the Academy of Sciences of
the Czech Republic is supported by RVO:67985840.}

\author{Martina Hofmanov\'a}
\address[M. Hofmanov\'a]{Fakult\"at f\"ur Mathematik, Universit\"at Bielefeld, D-33501 Bielefeld, Germany}
\email{hofmanova@math.uni-bielefeld.de}
\thanks{M.H. gratefully acknowledges the financial support by the German Science Foundation DFG via the Collaborative Research Center SFB1283.}

\begin{abstract}
We are concerned with the long time behavior of the stochastic Navier--Stokes system for  compressible fluids in dimension two and three.
In this setting, the part of the phase space occupied by the solution depends sensitively on the choice of the initial state. Our main results are threefold. (i) The kinetic energy of a solution is  universally and asymptotically bounded, independent of the initial datum. (ii) Time shifts of a solution with initially controlled energy are asymptotically compact and generate an entire solution defined for all $t\in R$. (iii) Every solution with initially controlled energy generates a stationary solution and even an ergodic stationary solution on the closure of the convex hull of its $\omega$--limit set.

\end{abstract}

\keywords{Navier--Stokes system, compressible fluid, stochastic forcing, stationary solutions}

\date{\today}

\maketitle

\tableofcontents
\section{Introduction}

The role of random forcing incorporated in originally deterministic models is in many cases to substitute for the 
effect of external driving mechanism represented by inhomogeneous boundary conditions. This gives rise to a mathematically 
simpler model that should retain, however, the essential behavior of the original system at least in the long run. In the framework of 
continuum fluid mechanics, equations with random forcing should shed some light on more complex problems related to turbulence.
In particular, the celebrated \emph{ergodic hypothesis} asserts:

\begin{quotation}
\emph{Time averages along trajectories of the flow converge, for large enough times, to an ensemble average given by a certain probability measure.}
\end{quotation} 

There is a common belief that such measure is in fact unique and completely characterizes the behavior of the fluid in the long run. This is  supported by the pieces of evidence in the case of \emph{incompressible} flows driven by the Navier--Stokes system with an additive  stochastic forcing. More precisely, for 
the incompressible planar flow driven by a physically relevant very degenerate stochastic forcing, unique ergodicity was established by
Hairer and Mattingly \cite{HaiMat}. The absence of a similar result in the physically relevant 3-D case is due to the existing gaps in the mathematical theory, in particular stability (uniqueness) of solutions with respect to the initial data. Nevertheless, it has been proved that  noise has a beneficial impact when it comes to long time behavior and ergodicity. Da Prato and Debussche \cite{DaPDeb} obtained a unique ergodicity for 3-D stochastic incompressible Navier--Stokes equations with non-degenerate  noise. The theory of Markov selections by Flandoli and Romito \cite{FlaRom} provides an alternative approach which also allowed to prove ergodicity for every Markov solution, see Romito \cite{MR2386571}.
The concept of (statistically) 
stationary solutions has been introduced both in the deterministic 
\cite{FMRT}, \cite{VisFur} and stochastic framework \cite{FlaGat}. 

The question of identifying  a \emph{unique} invariant measure for \emph{compressible} fluid motions excited by random forces is substantially different from the incompressible setting. As a matter of fact, this possibility is naturally limited/excluded 
as there are certain invariant quantities, for instance the total mass, that persist under the action of random forcing. 
Accordingly, the part of the phase space occupied by the trajectories necessarily depends sensitively on the choice of the initial state.
It is therefore desirable to show that  there exist
 invariant measures/stationary solutions generated by solutions to the initial value problem supported on a suitably defined $\omega-$limit set.

The main goal of the present paper is to investigate this question in the setting of the compressible Navier--Stokes system under stochastic perturbations.
Roughly speaking, the result can be shown by means of the standard Krylov--Bogoliubov theory as long as the solutions of the stochastic problem 
enjoy the following properties: 
\begin{itemize}
\item {\bf Global existence.} The problem admits a global--in--time solution -- a random process ranging a suitable phase space --
for any sample of initial data.   
 
\item {\bf Global boundedness.} The expected value of a suitable norm of global--in--time solutions is bounded independently of time.

\item {\bf Asymptotic compactness.} The law of any global--in--time solution is tight in a suitable space of trajectories. 

\end{itemize}

The above outlined points also summarize the strategy  of our proof. 
To be more precise, let $\vr = \vr(t,x)$ denote the mass density and $\vu = \vu(t,x)$ the bulk velocity  of a compressible viscous 
fluid occupying a bounded physical domain $Q \subset R^d$, $d=2,3$. 
In this paper, we are concerned with the \emph{compressible Navier--Stokes system} driven by a stochastic forcing:
\begin{equation} \label{E1}
\D \vr + \Div (\vr \vu) \dt = 0,
\end{equation}
\begin{equation} \label{E2}
\D (\vr \vu) + \Div (\vr \vu \otimes \vu) \dt + \Grad p(\vr) \dt = \Div \,\mathbb{S}(\Grad \vu) \dt + \vr \mathbf{g}(\vr,\vu) \D t+
 \vr\mathbb{F} (\vr,\vu)\D W,
\end{equation}
\begin{equation} \label{E3}
\mathbb{S}(\Grad \vu) = \mu \left( \Grad \vu + \Grad^t \vu - \frac{2}{d} \Div \vu \mathbb{I} \right)
+ \lambda \Div \vu \mathbb{I},\quad \mu > 0, \ \lambda \geq 0,
\end{equation}
where we include a deterministic  force $\mathbf{g}$ as well as a stochastic  force driven by a Wiener process $W$.
The problem is closed by imposing the no--slip boundary condition
\begin{equation} \label{E3a}
\vu|_{\partial Q} = 0.
\end{equation} 

We refer to \cite{BrFeHobook, BrFeHo2018C} for the existing mathematical theory of the problem \eqref{E1}--\eqref{E3a}. The long--time behavior of 
global--in--time solutions of the \emph{deterministic} problem was studied in \cite{FP9}, \cite{NOS1}, \cite{NOST}, and the monograph 
\cite{FeiPr}. In particular, the problem may admit several (a continuum of) stationary solutions already in the deterministic 
setting, see \cite{FP16}. In \cite{FanFeiHof}, it was proved that \emph{every} bounded solution to the deterministic system gives raise to a statistical stationary solution supported on its $\omega-$limit set.
The existence of \emph{stochastically} stationary solutions to \eqref{E1}--\eqref{E3a} with a given total mass was established in \cite{BrFeHo2017}.   

In the present paper, we go beyond the result of \cite{BrFeHo2017}. In particular, we focus on a  physically relevant \emph{hard sphere} pressure--density equation of state (see Section~\ref{s:2.1}) and show that solutions with initially controlled energy remain universally and asymptotically bounded in expectation, independently of the initial condition, see Theorem~\ref{thm:bound}. In other words, all such solutions ultimately  enter a bounded absorbing set. Furthermore, we establish an asymptotic compactness of solutions which in particular applies to sequences of time shifts of solutions and gives raise to an entire solution defined for all $t\in R$, see Theorem~\ref{thm:comp}. Finally, we deduce that  \emph{any}  solution with initially controlled energy generates a stationary solution,
  see Theorem~\ref{thm:main}. 

Unlike in the deterministic setting (cf.~\cite{FanFeiHof}), we define an $\omega$--limit set as a set of probability laws, not a set of trajectories. We then show that for \emph{any}  solution with initially controlled energy there is a stationary solution whose law belongs to the closure of the convex hull of its $\omega$--limit set, see Corollary~\ref{c:6.1}. Moreover, the method is constructive -- the solution is 
obtained by a direct application of Krylov--Bogoliubov's method applied on the $\omega$--limit set. Our approach is motivated 
by the pioneering work of It\^ o and Nisio \cite{ItNo} and the idea of Sell \cite{SEL}  replacing the natural phase space by the space of trajectories, see also Romito \cite{Romi}. Finally, we prove that there is an ergodic stationary solution on the closure of the convex hull of every such $\omega-$limit set, see Theorem~\ref{t:6.3}. In the setting of \cite{BrFeHo2017}, the existence of the $\omega-$limit set cannot even be proved and the procedure of the present paper cannot be repeated there.

As a byproduct of our strategy, we deduce two new results for the stochastic Navier--Stokes system \eqref{E1}--\eqref{E3a} that are
of independent interest themselves:

\begin{itemize}
\item {\bf Bounded moments of the total energy.} 
\[
\limsup_{t \to \infty} \expe{ \left( \intQ{ E(\vr, \vr \vu ) } \right)^m } \leq \mathcal{E}_\infty(m),\ m = 1,2,\dots 
\]
where 
\[
E(\vr, \vm) = \frac{1}{2} \frac{|\vm|^2}{\vr} + P(\vr),\ P'(\vr) \vr - P(\vr) = p(\vr), \ P(0) = 0 
\]
is the energy of the fluid. The constants $\mathcal{E}_\infty(m)$, $m=1,2,\dots$ are universal and independent of the initial condition.

\item {\bf Asymptotic compactness.} The law of the time shifts of a fixed solution 
\[
\mathcal{L} [\vr(\cdot + \tau_n), \vu(\cdot + \tau_n) ], \tau_n \to \infty
\] 
is tight in a suitable trajectory space.

\end{itemize}

In the context of stochastic incompressible Navier--Stokes system, the analogous results are basically immediate due to the dissipative nature of the problem and the good compactness properties of solutions. However, the compressible Navier--Stokes system is a mixed parabolic--hyperbolic system with a very delicate structure. Its incompressible counterpart instead is a semilinear parabolic system, to some extend rather similar to the heat equation.
As a consequence, the global boundedness as well as the asymptotic compactness both become substantially more difficult in the compressible case. The key difficulty is, on the one hand, the lack of energy dissipation stabilizing the system in the long run, on the other hand, the fact that the available energy and pressure estimates do not directly lead to strong convergence of the density necessary in order to pass to the limit.

 The first issue is overcome by performing a higher order energy estimate together with a new dissipation balance estimate and proving the existence of a bounded absorbing set in expectation, see Section~\ref{s:3.2}. 
The solution of the second issue leans on establishing the strong convergence of the approximate densities despite the fact that the initial conditions are lost in the limit process, see Section~\ref{s:4.4}. This is a delicate issue and requires careful analysis of the oscillation ``damping'' in the renormalized equation of continuity.

The paper is organized as follows. In Section \ref{sec:NS}, we recall the basic definitions, formulate the hypotheses and state the main results. The global--in--time estimates are established in Section \ref{i}. Section \ref{co} and Section~\ref{cs} form the heart of the paper. In Section \ref{co}, we show tightness of the time shifts of global trajectories. This property is subsequently used in Section 
\ref{cs}, where the Krylov--Bogoliubov method is applied to obtain a stationary solution. Section~\ref{s:erg} is devoted to the study of the ergodic structure of the set of stationary solutions. A sketch of the proof of existence of global--in--time solutions is given in the Appendix.

\section{Mathematical framework and main results}
\label{sec:NS}

\subsection{Pressure--density equation of state}\label{s:2.1}

The uniform bounds on the total energy require strong control of the fluid density. To this end, we consider the 
physically relevant \emph{hard sphere} pressure--density equation of state. Specifically, there is 
a limit density $\Ov{\vr} > 0$ such that 
\begin{equation} \label{p1}
\begin{split}
p &\in C^1[0,\Ov{\vr}) ,\ p(0)=0,\ p'(\vr) > 0 \ \mbox{for any}\ 0 < \vr < \Ov{\vr},\\
p'(\vr) &\geq a \vr^{\gamma - 1},\ a > 0,\ 
\lim_{\vr \to \Ov{\vr}-} (\Ov{\vr} - \vr)^\beta p(\vr) = \Ov{p} > 0,\\ 
&\mbox{for some exponents}\ \gamma > 1, \ \beta > 3. 
\end{split}
\end{equation}
The restriction on $\gamma$ and $\beta$ are technical and can be possibly relaxed. The essential feature of \eqref{p1} is the singularity 
of the pressure at $\Ov{\vr}$ yielding the (deterministic) bound $\vr \leq \Ov{\vr}$. This hypothesis is relevant for any \emph{real} fluid, see e.g. Carnahan and Starling \cite{CaSt}, Kolafa et al. \cite{KLM04AEHS}.

\subsection{Driving force}
\label{ss:force}

The deterministic driving force is given by  $\mathbf{g} \in C(\bar{Q}\times [0,\bar{\vr}]\times R^{d}; R^{d})$ satisfying
\begin{equation} \label{p12bis}
|\mathbf{g}(x,\vr,\vu)|\lesssim 1+|u|^{\alpha}\ \mbox{for some}\ \alpha\in [0,1).
\end{equation}

Let $T\geq0$ and let $\left( \Omega, \mathfrak{F}, (\mathfrak{F}_t)_{t \geq -T}, \mathbb{P} \right)
$ be a complete probability space
with a complete right-continuous filtration $(\mathfrak{F}_t)_{t \geq -T}$.
The  stochastic process $W$ is a cylindrical $(\mf_t)$-Wiener process on a separable Hilbert space $\mathfrak{U}$ normalized so that $W(0)=0$. It is formally given by the expansion $W(t)=\sum_{k=1}^\infty e_k\, W_k(t)$ where $(W_k)_{k\in\N}$ is a sequence of mutually independent real-valued  Wiener processes relative to $(\mf_t)_{t\geq -T}$ normalized so that $W_{k}(0)=0$, and $(e_k)_{k\in\N}$ is a complete orthonormal system in  $\mathfrak{U}$.
Accordingly, the diffusion coefficient ${\mathbb F}$ is defined as a superposition operator
$\mathbb F(\varrho, \vu):\mathfrak{U}\rightarrow L^1(\torN,R^{d})$,
$$
{\mathbb F}(\varrho,\vu)e_k=\bfF_k(\cdot,\varrho(\cdot),\vu(\cdot)).
$$
The coefficients $\bfF_{k} = \bfF_k (x, \varrho, \vu) :\torN \times[0,\Ov{\vr}] \times R^d \rightarrow R^d$ are $C^1$-functions such that it holds
\begin{equation} \label{p12}
| \vc{F}_k (x, \vr, \vu) | + | \nabla_{\vr,\vu} \vc{F}_k (x, \vr, \vu) | \leq f_k(1+|\vu|^{\alpha})\ \mbox{for some}\ \alpha \in [0,1),\quad 
\sum_{k \geq 1} f_k^2 < \infty,
\end{equation}
uniformly in $x\in\mt$.
Finally, we define the auxiliary space $\mathfrak{U}_0\supset\mathfrak{U}$ via
$$\mathfrak{U}_0=\bigg\{v=\sum_{k\geq 1}\alpha_k e_k;\;\sum_{k\geq 1}\frac{\alpha_k^2}{k^2}<\infty\bigg\},$$
endowed with the norm
$$\|v\|^2_{\mathfrak{U}_0}=\sum_{k\geq 1}\frac{\alpha_k^2}{k^2},\qquad v=\sum_{k\geq 1}\alpha_k e_k.$$
Note that the embedding $\mathfrak{U}\hookrightarrow\mathfrak{U}_0$ is Hilbert--Schmidt. Moreover, trajectories of $W$ are $\prst$-a.s. in $C_{\rm{loc}}([-T,\infty);\mathfrak{U}_0)$.

\subsection{Dissipative martingale solutions}

We give a definition of a solution to \eqref{E1}--\eqref{E3a}. For future use, it is convenient to consider a general 
time interval $[-T, \infty)$ with $T \geq 0$.

\begin{Definition}[Dissipative martingale solution]\label{def:sol}
The quantity  $$\big((\Omega,\mf,(\mf_t)_{t\geq -T},\prst),\varrho,\bfu,W)$$
is called a {\em dissipative martingale solution} to \eqref{E1}--\eqref{E3a} on the time interval $[-T, \infty)$, provided the following holds:
\begin{enumerate}[(a)]
\item $(\Omega,\mf,(\mf_t)_{t\geq -T},\prst)$ is a stochastic basis with a complete right-continuous filtration;
\item $W$ is a cylindrical $(\mf_t)$-Wiener process normalized so that $W(0) = 0$;
\item the density $0 \leq \vr \leq \Ov{\vr} $ belongs to the space
$
C_{\rm{weak,loc}}([-T,\infty); L^{q}(\mt))$
$ \mathbb{P}\mbox{-a.s.}$ for any $1 \leq q < \infty$ and is $(\mf_t)$-adapted;

\item the momentum $\vr\bfu$
belongs to the space
$
 C_{\rm{weak,loc}}([-T, \infty); L^2(\mt,R^{d}))
$
\pas\ and is $(\mf_t)$-adapted;
\item the velocity $\vu$ belongs to $
 L^2_{\rm{loc}}([-T , \infty); W^{1,2}_0(\mt,R^d))
$
\pas\ and is $(\mathfrak{F}_{t})$-adapted\footnote{Adaptedness of the velocity  is understood in the sense of random distributions, cf. \cite[Chapter~2.8]{BrFeHobook}.};
\item the total energy
\[
t \mapsto \int_{\mt} E \left( \vr, \vr \vu \right)(t, \cdot) \dx
\]
belongs to the space
$
L^\infty_{\rm{loc}}[-T,\infty)
$
\pas;
\item the equation of continuity
\begin{equation*}
\left[ \int_{\mt} \vr \psi \ \dx \right]_{t = \tau_1}^{t = \tau_2} -
\int_{\tau_1}^{\tau_2} \int_{\torN}\varrho\bfu\cdot\Grad\psi\dxt=0
\end{equation*}
holds for all $-T \leq \tau_1 < \tau_2$, $\psi\in C^1(\Ov{Q})$, $\prst$-a.s.;
\item for any $b\in C^1(\R)$ 
\begin{equation} \label{renorm}
\begin{split}
&\left[ \int_{\mt} b(\vr) \psi \ \dx \right]_{t = \tau_1}^{t = \tau_2}\\
 &- \int_{\tau_1}^{\tau_2} \int_{\torN} b(\varrho)\bfu\cdot\Grad\psi \,\dif x\,\dif t+
\int_{\tau_1}^{\tau_2} \int_{\torN} \big(b'(\varrho)\varrho-b(\varrho)\big)\diver\bfu\,\psi \,\dif x\,\dif t=0.
\end{split}
\end{equation}
for all $- T \leq \tau_1 < \tau_2$, $\psi \in C^1(\Ov{Q})$, $\prst$-a.s.;
\item the momentum equation
\begin{align}
\nonumber
 \left[ \intTN{ \vr \vu \cdot \bfphi } \right]_{t = \tau_1}^{t= \tau_2}
 &-\int_{\tau_1}^{\tau_2} \intTN{ \Big[ \vr \vu \otimes \vu : \Grad \bfphi  +  p(\vr) \Div \bfphi  \Big]  }
\dt\\
&+\int_{\tau_1}^{\tau_2} \  \intTN{\mathbb{S}(\Grad \vu) : \Grad \bfphi  } \dt\nonumber\\
&\hspace{-1.5cm}= 
\int_{\tau_1}^{\tau_2} \intQ{ \vr \vc{g} (\vr, \vr \vu)  \cdot \bfphi} \dt  + \sum_{k=1}^\infty\int_{\tau_1}^{\tau_2}  \left( \intTN{  \vr{\vc{F}_k} (\vr, \vr\vu) \cdot \bfphi } \right) \, \D  W_k
\label{N2}
\end{align}
holds for all $- T \leq \tau_1 < \tau_2$,  $\bfvarphi\in C^1_c(\torN;R^d)$, $\prst$-a.s.;
\item
the energy inequality
\begin{align} \nonumber
- &\int_{-T}^\infty \partial_t \phi \intQ{ E(\vr, \vr \vu) } \dt 
+ \int_{-T}^\infty \phi 
\intQ{ \mathbb{S}(\Grad \vu) : \Grad \vu }  \dt \nonumber
\\
& \leq \int_{-T}^\infty \phi \intQ{\vr \vc{g} (\vr, \vr \vu) \cdot \vu     } \dt + \frac{1}{2} 
\int_{-T}^\infty \phi 
\sum_{k=1}^\infty \intQ{ \vr |\vc{F}_k(\vr,\vr\bfu)|^2 }  \dt \nonumber \\
&\qquad+ \sum_{k = 1}^\infty\int_{-T}^\infty \phi  
\left( \intQ{ \vr\vc{F}_k(\vr,\vr\bfu) \cdot \vu }\right) {\rm d}W_k,
\label{N3}
\end{align}
holds for all $\phi\in C^1_c((-T, \infty))$, $\phi \geq 0$, $\prst$-a.s.
\end{enumerate}
\end{Definition}

We
recall that the total energy 
\[
E (\vr, \vm) = \left\{ \begin{array}{l} \frac{1}{2} \frac{|\vm|^2}{\vr} + P(\vr) \ \mbox{for}\ \vr > 0, 
\\ 0 \ \mbox{if}\ \vr = 0, \ \vm = 0  \\ \infty  \ \mbox{otherwise} \end{array} \right.
\]
with the pressure potential $P$ defined through $P'(\vr) \vr - P(\vr) = p(\vr), \ P(0) = 0$,
is a convex lower semi-continuous function on $[0,\overline\vr)\times R^{d}$; whence  
\[
t \mapsto \intQ{ E(\vr, \vr \vu) (t) } 
\]
is a lower semi-continuous function defined for any $t \in [-T , \infty)$ $\prst$-a.s. for any dissipative martingale solution $(\vr, \vu)$. 
In addition, it follows from the energy inequality \eqref{N3} that the limit
\begin{equation}\label{conv1}
{ \mathcal{E} }(t) = \lim_{\delta \to 0+} \frac{1}{\delta} \int_{t}^{t + \delta} \intQ{ E(\vr, \vr \vu) (s, \cdot) } {\rm d}s 
\end{equation}
is well defined and $\mathfrak{F}_{t}-$measurable for any $t \geq -T$. Moreover, the 
function ${ \mathcal{E} }$ is c\` adl\` ag in $[0, \infty)$, thus 
progressively measurable, and 
\begin{equation} \label{conv}
\mathcal{E}(t) = \intQ{ E(\vr, \vr \vu) (t, \cdot) } \ \mbox{a.a. in}\ (-T, \infty) \ \prst-\mbox{a.s.}
\end{equation}
In the remaining part of the paper, we use $\mathcal{E}$ to denote the total energy keeping in mind \eqref{conv}.

\subsection{Stationary solutions}
The concept of stationary solution is motivated by the approach of It\^ o and Nisio \cite{ItNo}. We first introduce the space of trajectories. 
Despite the fact that the (weakly) time continuous quantities are the conservative variables $\vr$, $\vm = \vr \vu$, it is more convenient to consider the \emph{standard variables} 
\[
\begin{split}
\vr &\in C_{\rm weak,loc}(R; L^q(Q)),\ 1 \leq q < \infty,\\ 
\vu &\in L^2_{\rm loc}(R; W^{1,2}_0 (Q; R^d) ), 
\end{split}
\]
together with the noise 
\[
 W \in C_{\rm loc,0}(R; \mathfrak{U}_0 ), 
\] 
where $ C_{\rm loc,0}$ denotes the space of continuous functions vanishing at $0$.
We define the   trajectory space 
\[
\mathcal{T} = C_{\rm weak, loc}(R; L^q(Q)) 
\times \left( L^2_{\rm loc}(R; W^{1,2}_0 (Q; R^d) ), w \right) \times C_{\rm loc,0}(R; \mathfrak{U}_0 ).
\]

\begin{Definition}\label{def:ent}
We say that $\big((\Omega,\mf,(\mf_t)_{t\in R},\prst),\varrho,\bfu,W)$ is an entire solution of the problem \eqref{E1}--\eqref{E3a} if  $(\vr,\vu, W)\in\mathcal{T}$ \pas\ and if $\big((\Omega,\mf,(\mf_t)_{t\geq-T},\prst),\varrho,\bfu,W)$  is  a dissipative martingale solution on $[-T,\infty)$ for all $T\geq 0$ in the sense of Definition~\ref{def:sol}.
\end{Definition}

Next, we introduce the time shift operator, 
\[
\mathcal{S}_\tau [\vr, \vu,  W] (t) = 
[\vr (t + \tau), \vu(t + \tau),  W (t + \tau)-W(\tau)], \ t \in R, \ \tau \in R.
\]
It is easy to check that the time shift  
\[
\mathcal{S}_\tau [\vr, \vu,  W]  ,\ \tau \geq 0,
\]
of any dissipative martingale solution 
$$\big((\Omega,\mf,(\mf_t)_{t\geq-T},\prst),\varrho,\bfu,W)$$
on $[-T,\infty)$ gives rise to another dissipative  martingale solution 
$$\big((\Omega,\mf,(\mf_{t + \tau})_{t\geq-T-\tau},\prst),\varrho (\cdot + \tau) ,\bfu (\cdot + \tau),W (\cdot + \tau) - 
W(\tau) \big)$$ 
on $[-T-\tau,\infty)$
of the same problem.

\begin{Definition}[Stationary solution] \label{def:stationary}

We say that an entire solution
$$\big((\Omega,\mf,(\mf_t)_{t\in R},\prst),\varrho,\bfu,W)$$ of the problem \eqref{E1}--\eqref{E3a} is \emph{stationary} if its law 
\[
\mathcal{L}_{\mathcal{T}}[\vr, \vu,  W] 
\] 
is shift invariant in the trajectory space $\mathcal{T} $,
meaning 
\[
\mathcal{L}_{\mathcal{T}} \left[ \mathcal{S}_\tau [\vr, \vu,  W] \right] = \mathcal{L}_{\mathcal{T}}[\vr, \vu,  W]
\ \mbox{for any}\ \tau \in R.
\]

\end{Definition}

\begin{Remark}\label{r:2.4}
It is convenient to regard solutions on $[-T,\infty)$ as trajectories in $\mathcal{T}$. To this end, we tacitly extend $[\vr,\vu,W](t)=[\vr,\vu, W](-T)$ for all $t\leq -T$.
\end{Remark}

\subsection{Main results}
\label{s:2.5}

Having collected all the necessary material we are ready to state our main results.

\begin{Theorem}[Stationary solutions generated by bounded trajectories] \label{thm:main}

Let $Q \subset R^d$, $d = 2,3$ be a bounded Lipschitz domain. Suppose that the pressure $p$, the deterministic 
driving force $\vc{g}$, and the noise diffusion coefficients $\mathbb{F}$ satisfy the hypotheses \eqref{p1}, 
\eqref{p12bis}, and \eqref{p12}, respectively. Let 
$$\big((\Omega,\mf,(\mf_t)_{t\geq0},\prst),\varrho,\bfu,W)$$ be a dissipative martingale solution of the problem 
\eqref{E1}--\eqref{E3a} specified in Definition~\ref{def:sol} such that 
\begin{equation} \label{hypo1}
\begin{aligned}
&\expe{ \mathcal{E}(0)^4 } < \infty,\ \prst \left[ \Ov{\vr} - \frac{1}{|Q|}\intQ{ \vr(0, \cdot)} > \delta  \right] = 1,\\
&\qquad\ \prst \left[ \frac{\vr\vu(0)}{\vr(0)}\in W^{1,2}_{0}(Q,R^{d})\right] = 1,
\end{aligned}
\end{equation}
for some deterministic constant $\delta > 0$.

Then there is a sequence $T_n \to \infty$ and a stationary solution 
$$\big((\tilde{\Omega},\tilde{\mf},(\tilde{\mf}_t)_{t\in R},\tilde{\prst}),\tvr,\tilde{\vu}, 
\tilde{W})$$ 
such that 
\[
\frac{1}{T_n} \int_0^{T_n} \mathcal{L}_{\mathcal{T}}\left[ \mathcal{S}_t \left[ \vr, {\vu},  {W} \right] \right] \dt \to 
\mathcal{L}_{\mathcal{T}} \left[ \tvr, \tilde{\vu}, \tilde{W} \right] \ \mbox{narrowly as}\ 
n \to \infty.
\]
\end{Theorem}

The proof of Theorem \ref{thm:main} leans on two auxiliary results that are of independent interest. 

\begin{Theorem}[Ultimate boundedness] \label{thm:bound}

Under the hypotheses of Theorem \ref{thm:main}, there exists a universal constant $\mathcal{E}_\infty (m)$ such that 
\[
\limsup_{t \to \infty} \expe{ \left( \intQ{ E(\vr, \vr \vu)(t, \cdot) } \right)^m } \leq \mathcal{E}_\infty (m)
\]
for any dissipative martingale solution  
$$\big((\Omega,\mf,(\mf_t)_{t\geq0},\prst),\varrho,\bfu,W)$$ of the problem 
\eqref{E1}--\eqref{E3a} with the initial data satisfying
\[
\expe{  \mathcal{E}(0)^m } < \infty, \ \prst \left[ \Ov{\vr} - \frac{1}{|Q|}\intQ{ \vr(0, \cdot)} > \delta  \right] = 1
\]
for some $m \geq 1$, $\delta > 0$. More precisely, there exist universal constants $c_{m,1},c_{m,2},D_{m}>0$ such that
\begin{equation}\label{eq:61}
\expe{ \mathcal{E}(t)^m } \leq \exp( - D_m t) \left( \expe{ \mathcal{E}(0)^m} + c_{m,{1}}\right) + c_{m,2}
\end{equation}
for all $t>0$. 

\end{Theorem}

In the following result, we employ the notation $\mathcal{E}_{n}$ for the c\`adl\`ag version of the energy associated to $(\vr_{n},\vr_{n}\vu_{n})$ and defined through \eqref{conv1}.

\begin{Theorem}[Asymptotic compactness] \label{thm:comp}
Under the hypotheses of Theorem \ref{thm:main}, let 
\[
\Big((\Omega_n,\mf_n,(\mf_n )_{t\geq - T_{n}},\prst_n),\varrho_n,\bfu_n ,W_n 
\Big),\ n=1,2,\dots 
\]
be a sequence of dissipative martingale solutions of the problem \eqref{E1}--\eqref{E3a} 
on the time interval $[-T_{n}, \infty)$, $T_{n} \to \infty$, 
such that 
\begin{equation}\label{eq:ass}
\begin{aligned}
\sup_{n \geq 1}
&\expe{  \mathcal{E}_{n}(-T_{n})^{4}  } < \infty,
\ \prst \left[ \Ov{\vr} - \frac{1}{|Q|}\intQ{ \vr_n(-T_{n})} > \delta  \right] = 1,\\
&\prst \left[ \frac{\vr_{n}\vu_{n}(-T_{n})}{\vr_{n}(-T_{n})}\in W^{1,2}_{0}(Q,R^{d})\right] = 1, \ n = 1,2,\dots
\end{aligned}
\end{equation}
for some deterministic constant $\delta > 0$.

Then there exists a subsequence (not relabeled) such that
\[
\mathcal{L}_{\mathcal{T}} \left[ \vr_n, \vu_n,  W_n  \right] \to 
\mathcal{L}_{\mathcal{T}} \left[  \vr, \vu,  W  \right]\ \mbox{narrowly as}\ n \to \infty, 
\]
where $(\vr, \vu, W)$ is an entire  solution of the same problem defined on a certain probability space 
\[
(\Omega, \mathfrak{F}, \prst)\ \mbox{with a filtration}\ (\mathfrak{F}_t)_{t\in R}.
\]
\end{Theorem}

The energy momenta estimates claimed in Theorem \ref{thm:bound} are new in the context of the stochastic problem and depend essentially 
on the properties of the hard--sphere pressure equation of state in \eqref{p1}. The crucial point is to control the density (pressure) in terms 
of the dissipation term 
\[
\intQ{ \mathbb{S} (\Grad \vu): \Grad \vu }.
\] 
With these estimates at hand, the proof of Theorem \ref{thm:comp} follows the steps of the proof of existence. There is, however, 
a key  difference, namely, the initial data represented by the value of the time shifts at the time $-T_{n}$ are ``lost'' in the limit process. In particular, the crucial ingredient of the existence proof -- compactness of the initial data -- is no longer available. Instead, the deterministic argument on propagation of density oscillations proved originally in \cite{FP15} must be adapted to the stochastic framework. 

\section{Global in time estimates}
\label{i}

Our goal is to show the momentum estimates claimed in Theorem \ref{thm:bound}.

\subsection{A Gronwall-type estimate for BV-functions}
 
Let us start with a standard version of Gonwalls' lemma.

\begin{Lemma} \label{uL1}
Let $F \in BV_{\rm loc}(0, \infty)$ be such that 
\begin{equation} \label{u10}
F(\tau_2 + ) + D \int_{\tau_1}^{\tau_2} F(t) \ \dt \leq F(\tau_1 -) + C (\tau_2 - \tau_1) 
\end{equation}
for any $0 < \tau_1 \leq \tau_2$ with some $D>0$. 
Then 
\[
F(t\pm) \leq \exp (- Dt) \left( F(0+) - \frac{C}{D} \right) + \frac{C}{D}.
\] 
for all $t>0$
\end{Lemma}

\begin{proof}
we first note that \eqref{u10} is equivalent to
\begin{equation*} 
G(\tau_2+) + D \int_{\tau_1}^{\tau_2} G(t) \ \dt \leq G(\tau_1 -),\quad G(t)=F(t)-\frac{C}{D}.
\end{equation*}
It follows that the function 
\[
\exp (Dt)G(t)= \exp (Dt)\left( F(t) - \frac{C}{D} \right) 
\]
is non--increasing on $(0, \infty)$; whence 
\[
F(t\pm) \leq \exp (- Dt) \left( F(0+) - \frac{C}{D} \right) + \frac{C}{D},\ t > 0.
\] 

\end{proof}

\subsection{Higher energy moments}\label{s:3.2}

Let us introduce the so--called Bogovskii operator $\mathcal{B}$ enjoying the following properties: 
\begin{align}\label{def:bog}
\begin{aligned}
\mathcal{B} : L^q_0 (Q) &\equiv \left\{ f \in L^q(Q) \ \Big| \ \int_Q f \dx = 0 \right\}
\to W^{1,q}_0(Q, R^d),\ 1 < q < \infty,\\
\Div \mathcal{B}[f] &= f ,\\ 
\| \mathcal{B}[f] \|_{L^r(Q)} &\aleq \| \vc{g} \|_{L^r(Q; R^d)} \ \mbox{if}\ 
f = \Div \vc{g}, \ \vc{g} \cdot \vc{n}|_{\partial Q} = 0, \ 1 < r < \infty,
\end{aligned}
\end{align}
see Galdi \cite[Chapter 3]{GALN}, Gei{\ss}ert, Heck, and Hieber \cite{GEHEHI}. Here $\vc{n}$ denotes the outer unit normal on $\partial Q$.

In what follows, we neglect the deterministic forcing $\vr \vc{g} (\vr, \vr \vu)$ as its treatment does not present any additional 
difficulties.
Recalling our convention \eqref{conv} we deduce from the energy inequality \eqref{N3}
\begin{equation} \label{i3}
\begin{split}
&\Big[ \mathcal{E}(t) \Big]_{t = \tau_1}^{t = \tau_2} + \int_{\tau_1}^{\tau_2} \intQ{ \mathbb{S}(\Grad \vu) : \Grad \vu } \dt 
\\ &\quad\leq \frac{1}{2} \int_{\tau_1}^{\tau_2} \intQ{ \vr \sum_{k \geq 1} |\vc{F}_k (\vr,  \vu)|^2 } \dt
+ \sum_{k\geq 1}\int_{\tau_1}^{\tau_2} \left( \intQ{ \vr \vc{F}_{k}(\vr, \vu) \cdot \vu } \right) \D W_{k}
\end{split}
\end{equation}
for any $0 \leq \tau_1 < \tau_2$ $\prst$-a.s.

Next, we use the quantity 
\[
\mathcal{B} \left[ \beta(\vr) - \frac{1}{|Q|} \intQ{ \beta(\vr) } \right]  
\]
as a test function in the momentum balance \eqref{N2}. This is not completely obvious as our test function is a random 
variable, however, the function $\beta (\vr)$ satisfies the deterministic equation \eqref{renorm} and such a step can be 
rigorously justified 
by the arguments detailed in \cite[Section 4.4.2]{BrFeHobook} by the use of the generalized It\^{o} formula \cite[Theorem~A.4.1]{BrFeHobook}.  
\begin{equation} \label{i5}
\begin{split}
\int_{\tau_1}^{\tau_2} &\intQ{ p(\vr) \beta(\vr) } \dt =
\left[ \intQ{ \vr \vu \cdot \mathcal{B} \left[ \beta(\vr) - \frac{1}{|Q|} \intQ{ \beta(\vr) } \right] }  \right]_{t= \tau_1}^{t = \tau_2}
\\&+ \frac{1}{|Q|} \int_{\tau_1}^{\tau_2} \left( \intQ{ p(\vr) } \intQ{ \beta(\vr) } \right) \dt \\
&- \int_{\tau_1}^{\tau_2} \intQ{  \vr \vu \otimes \vu : \Grad \mathcal{B} \left[ \beta(\vr) - \frac{1}{|Q|} \intQ{ \beta(\vr) }  \right] } \dt\\
&+ \int_{\tau_1}^{\tau_2} \intQ{  \mathbb{S} ( \Grad \vu ) : \Grad \mathcal{B} \left[\beta(\vr) - \frac{1}{|Q|} \intQ{ \beta(\vr) }  \right] } \dt \\
& - \int_{\tau_1}^{\tau_2} \intQ{ \vr \vu \cdot \mathcal{B} \Big[ \Div (\beta(\vr) \vu) \Big] } \dt
\\
&+ \int_{\tau_1}^{\tau_2} \intQ{ \vr \vu \cdot \mathcal{B} \left[ \left( \beta(\vr) - \beta'(\vr) \vr \right) \Div \vu 
- A \right] } \dt
\\
&- \sum_{k\geq 1}\int_{\tau_1}^{\tau_2} \left( \intQ{  \vr \vc{F}_{k}(\vr, \vu) \cdot \mathcal{B} \left[ \beta(\vr) - \frac{1}{|Q|} \intQ{ \beta(\vr) }  \right] } \right) \D W_{k}
\end{split}
\end{equation}
for any $0 \leq \tau_1 < \tau_2$ $\prst$-a.s., where $A=\frac{1}{|Q|}\intQ{ \left( \beta(\vr) - \beta'(\vr) \vr \right) \Div \vu } $.

We combine the energy inequality \eqref{i3} with the pressure estimates \eqref{i5} to obtain the total dissipation balance that is a crucial tool in the subsequent analysis. First, consider $\beta (\vr) = \vr$ in \eqref{i5} to gain $\prst$-a.s.
\begin{align} \label{i6}
\begin{aligned}
\int_{\tau_1}^{\tau_2}
&\int_Q p(\vr) \left( \vr - \frac{M}{|Q|} \right) \dx \dt \\&=\bigg[ \intQ{ \vr \vu \cdot \mathcal{B} \bigg[ \vr - \frac{M}{|Q|}  \bigg] }  \bigg]_{t= \tau_1}^{t = \tau_2} 
- \int_{\tau_1}^{\tau_2} \intQ{  \vr \vu \otimes \vu : \Grad \mathcal{B} \left[ \vr - \frac{M}{|Q|}  \right] }    \dt\\
&+ \int_{\tau_1}^{\tau_2} \intQ{  \mathbb{S} ( \Grad \vu ) : \Grad \mathcal{B} \left[ \vr - \frac{M}{|Q|}   \right] } \dt - \int_{\tau_1}^{\tau_2} \intQ{ \vr \vu \cdot \mathcal{B} \Big[ \Div (\vr \vu) \Big] } \dt\\&-\sum_{k\geq 1} \int_{\tau_1}^{\tau_2} \left( \intQ{  \vr \vc{F}_{k}(\vr, \vu) \cdot \mathcal{B} \left[ \vr - \frac{M}{|Q|}   \right] } \right) \D W_{k},
\end{aligned}
\end{align}
where 
\[
M = \intQ{ \vr(0, \cdot) }.
\]

Next, we recall the uniform bounds on the density and hypothesis \eqref{hypo1} implying
\begin{equation} \label{i7}
0 \leq \vr \leq \Ov{\vr},\ \frac{1}{|Q|}\intQ{ \vr } = \frac{M}{|Q|} \leq \Ov{\vr} - \delta\   \quad \p\mbox{-a.s.}
\end{equation} 
for some deterministic constant $\delta > 0$.
Setting 
\[
r = \frac{1}{2} \left( \frac{M}{|Q|} + \overline{\vr} \right) < \Ov{\vr} - \frac{\delta}{2}
\]
we deduce 
\[
\begin{split}
&\intQ{ p(\vr) \left (\vr - \frac{M}{|Q|} \right) } =
\int_{Q\cap[\vr \geq r]} p(\vr) \left (\vr - \frac{M}{|Q|} \right) \dx + \int_{Q\cap[\vr < r]}
p(\vr) \left (\vr - \frac{M}{|Q|} \right) \dx \\
&\qquad\geq \frac{\delta}{2} \int_{\vr \geq r} p(\vr) \ \dx - |Q| p(r) r 
\geq \frac{\delta}{2} \intQ{ p(\vr) } - |Q| \left( r p(r) + \frac{\delta}{2} p(r) \right). 
\end{split}
\]
Consequently, it follows from \eqref{i6}  $\prst$-a.s.
\begin{align} \label{i11}
\begin{aligned}
\frac{\delta}{2} &\int_{\tau_1}^{\tau_2} \intQ{ p(\vr) } \dt \\&\leq 
\left[ \intQ{ \vr \vu \cdot \mathcal{B} \left[ \vr - \frac{M}{|Q|}  \right] }  \right]_{t= \tau_1}^{t = \tau_2} - \int_{\tau_1}^{\tau_2} \intQ{  \vr \vu \otimes \vu : \Grad \mathcal{B} \left[ \vr - \frac{M}{|Q|}  \right] }    \dt
\\&+ \int_{\tau_1}^{\tau_2} \intQ{  \mathbb{S} ( \Grad \vu ) : \Grad \mathcal{B} \left[ \vr - \frac{M}{|Q|}   \right] } \dt - \int_{\tau_1}^{\tau_2} \intQ{ \vr \vu \cdot \mathcal{B} \Big[ \Div (\vr \vu) \Big] } \dt
\\
&- \sum_{k\geq 1}\int_{\tau_1}^{\tau_2} \left( \intQ{  \vr \vc{F}_{k}(\vr, \vu) \cdot \mathcal{B} \left[ \vr - \frac{M}{|Q|}   \right] } \right) \D W_{k} 
+ c (\tau_2 - \tau_1),
\end{aligned}
\end{align}
where $c$ is a deterministic constant. 

Next, we are going to estimate the second, third and fourth term on the right-hand side of \eqref{i11} by dissipation.
By virtue of the Korn--Poincar\' e inequality, 
\begin{equation} \label{i9}
\| \vu \|^2_{W^{1,2}_0 (Q; R^d)} \leq c_{KP} \intQ{ \mathbb{S}(\Grad \vu) : \Grad \vu }
\end{equation}
and the uniform bound \eqref{i7} we obtain
\begin{equation} \label{i10}
\intQ{ \vr |\vu|^2 } \leq\,c\, \intQ{ \mathbb{S}(\Grad \vu) : \Grad \vu }.
\end{equation}
Similarly, the Korn--Sobolev inequality yields
\begin{equation} \label{i10b}
\bigg(\intQ{ |\vu|^6 }\bigg)^{\frac{1}{3}} \leq\,c_{KS} \, \intQ{ \mathbb{S}(\Grad \vu) : \Grad \vu } 
\end{equation}
if $d=3$. Of course, the estimate holds for arbitrary exponent $q$ if $d=2$.
Consequently, we have $\prst$-a.s.
\begin{align*}
 \intQ{\vr \vu \otimes \vu : \Grad \mathcal{B} \left[ \vr - \frac{M}{|Q|}  \right] }&\lesssim \bigg(\intQ{ |\bfu|^6 }\bigg)^{\frac{1}{3}}\lesssim \intQ{ \mathbb{S}(\Grad \vu) : \Grad \vu },\\
 \intQ{  \mathbb{S} ( \Grad \vu ) : \Grad \mathcal{B} \left[ \vr - \frac{M}{|Q|}   \right] }&\lesssim\,\intQ{\vr^2}+\intQ{ \mathbb{S}(\Grad \vu) : \Grad \vu }\\
 &\lesssim\,1+\intQ{ \mathbb{S}(\Grad \vu) : \Grad \vu },\\
 \intQ{ \vr \vu \cdot \mathcal{B} \Big[ \Div (\vr \vu) \Big] }&\lesssim \intQ{\vr^2|\bfu|^2}\lesssim\intQ{ \mathbb{S}(\Grad \vu) : \Grad \vu },
\end{align*}
using properties of the Bogovskii operator, see \eqref{def:bog}, and boundedness of $\varrho$.
Plugging this into \eqref{i11} we conclude that $\prst$-a.s.
\begin{align}
\begin{aligned}
\frac{\delta}{2} \int_{\tau_1}^{\tau_2} \int_Q p(\vr) &\dx \dt - 
\left[ \intQ{ \vr \vu \cdot \mathcal{B} \left[ \vr - \frac{M}{|Q|}  \right] }  \right]_{t= \tau_1}^{t = \tau_2} \\&\leq
c_1 \int_{\tau_1}^{\tau_2} \intQ{ \mathbb{S}(\Grad \vu) : \Grad \vu } \dt 
+ c_2 (\tau_2 - \tau_1)\\
&- \sum_{k\geq 1}\int_{\tau_1}^{\tau_2} \left( \intQ{  \vr \vc{F}_{k}(\vr, \vu) \cdot \mathcal{B} \left[ \vr - \frac{M}{|Q|}   \right] } \right) \D W_{k}  
\end{aligned}\label{i12}
\end{align}
with deterministic constants $c_1$, $c_2$. 
Multiplying \eqref{i12} by a sufficiently small (deterministic) constant $\ep > 0$ and adding the resulting expression to 
the energy inequality \eqref{i3} we obtain the dissipation balance  $\prst$-a.s.
\begin{align} 
\nonumber
\Big[ \mathcal{E}(t)-& \ep\intQ{ \vr \vu \cdot \mathcal{B} \left[ \vr - \frac{M}{|Q|}  \right] } \Big]_{t = \tau_1}^{t = \tau_2} + 
\frac{1}{2} \int_{\tau_1}^{\tau_2} \intQ{ \left[ \mathbb{S}(\Grad \vu) : \Grad \vu + \ep \delta  p(\vr) \right] } \dt 
\\ &\leq C (\tau_2 - \tau_1)- \sum_{k\geq 1}\int_{\tau_1}^{\tau_2} \left( \intQ{ \vr \vc{F}_{k}(\vr, \vu) \cdot \left( \vu 
- \mathcal{B} \left[ \vr - \frac{M}{|Q|}  \right]  \right) } \right) \D W_{k},
\label{i13}
\end{align}
where $C$ is a deterministic constant, and where we have used 
\begin{align} \label{i14}
\begin{aligned}
\intQ{ \vr \sum_{k\geq 1} |\vc{F}_k (\vr, \vu )|^2 }&\leq\sum_{k\geq 1} f_k^2\intQ{\vr (1+|\vu|^{2\alpha}) } \\
& \leq  \frac14\intQ{  \mathbb{S}(\Grad \vu) : \Grad \vu   }+c,  
 \end{aligned}
\end{align}
with a deterministic constant $c$, cf. \eqref{p12}.

As the next step we aim at deriving higher moment estimates for the process
\[
\mathcal{D}(\tau) = \mathcal{E}(\tau)- \ep\intQ{ \vr \vu \cdot \mathcal{B} \left[ \vr - \frac{M}{|Q|}  \right] (\tau, \cdot)}
\]
in the spirit of It\^{o}'s formula applied to \eqref{i13}. 
For $\tau_1 > 0$ arbitrary, 
we consider the continuous process
\[
\begin{split}
\mathcal{D}_C (\tau) = &- \frac{1}{2} \int_{\tau_1}^{\tau} \intQ{ \Big[ \mathbb{S}(\Grad \vu) : \Grad \vu 
+ \varepsilon\delta  p(\vr) \Big]} \dt + C (\tau - \tau_1) \\ &-\sum_{k\geq 1} \int_{\tau_1}^{t} \left( \intQ{ \vr \vc{F}_{k}(\vr, \vu) \cdot \left( \vu 
- \mathcal{B} \left[ \vr - \frac{M}{|Q|}  \right]  \right) } \right) \D W_{k}.
\end{split}
\]
In accordance with the dissipation balance \eqref{i13}, the process 
\[
\mathcal{D}_M = \mathcal{D} - \mathcal{D}_C 
\]
is non--increasing in $(\tau_1, \infty)$. 

In order to apply It\^{o}'s formula to the process
$\mathcal{D}=\mathcal{D}_M+\mathcal{D}_C$ it suffices to regularise $\mathcal{D}_M$ in time. To avoid problems with progressive measurability
we introduce the backward regularization of a function $\mathcal{F} = \mathcal{F}(t)$ given by
\[
\mathcal{F}^\kappa (t) = \int_{-\kappa}^0 \mathcal{F}(t- s) \chi_\kappa(s) \D s, \ t > \kappa, 
\]
where $\chi_\kappa$ is a standard family of regularizing kernels. Applying the standard 
It\^o formula to $\mathcal{D}_C+\mathcal{D}_M^\kappa$  we get 
\begin{align}\label{eq:2810}
\begin{aligned}
\D \Phi ( \mathcal{D}_C &+ \mathcal{D}_M^\kappa ) = - \frac{1}{2} \Phi' \left( \mathcal{D}_C + \mathcal{D}_M^\kappa \right) 
\left( \intQ{ \Big[ \mathbb{S}(\Grad \vu) : \Grad \vu + \ep \delta  p(\vr) \Big]  } 
\right) \dt  \\ &+ C \Phi' \left( \mathcal{D}_C + \mathcal{D}_M^\kappa \right) \dt +
\Phi' \left( \mathcal{D}_C + \mathcal{D}_M^\kappa \right) \D \mathcal{D}^\kappa_M \\
&+ \frac{1}{2} \Phi'' \left( \mathcal{D}_C + \mathcal{D}_M^\kappa \right) 
\sum_{k\geq 1} \left( \intQ{ \vr \vc{F}_k(\vr, \vu) \cdot \left( \vu 
- \mathcal{B} \left[ \vr - \frac{M}{|Q|}  \right] \right) } \right)^2 \dt\\
&- \Phi' \left( \mathcal{D}_C + \mathcal{D}_M^\kappa \right) \sum_{k\geq 1}\left( \intQ{ \vr \vc{F}_{k}(\vr, \vu) \cdot \left( \vu 
- \mathcal{B} \left[ \vr - \frac{M}{|Q|}  \right]  \right) } \right) \D W_{k}
\end{aligned}
\end{align}
for any $\Phi \in C^2$. If, in addition, $\Phi' \geq 0$, we have 
\[
\Phi' \left( \mathcal{D}_C + \mathcal{D}_M^\kappa \right) \D \mathcal{D}^\kappa_M \leq 0
\]
using that $\mathcal{D}^\kappa_M$ is non-increasing. All other terms in \eqref{eq:2810}
converge to their counterparts as $\kappa \to 0$ and we obtain
\begin{equation} \label{m1a}
\begin{split}
\Big[ \Phi &\left( \mathcal{D}  \right) \Big]_{t = \tau_1 -}^{t = \tau_2-}+ \frac{1}{2} \int_{\tau_1}^{\tau_2} \Phi' \left( \mathcal{D} \right) 
\left( \intQ{ \Big[ \mathbb{S}(\Grad \vu) : \Grad \vu + \ep \delta  p(\vr) \Big]} \right) \dt\\ &\leq  
C \int_{\tau_1}^{\tau_2} \Phi' \left( \mathcal{D} \right)  \dt\\
&+ \frac{1}{2} \int_{\tau_1}^{\tau_2} \Phi'' \left( \mathcal{D} \right) 
\sum_{k \geq 1} \left( \intQ{ \vr \vc{F}_k(\vr, \vu) \cdot \left( \vu 
- \mathcal{B} \left[ \vr - \frac{M}{|Q|}  \right] \right) } \right)^2 \dt\\
&- \sum_{k\geq 1}\int_{\tau_1}^{\tau_2} \Phi' \left( \mathcal{D} \right) \left( \intQ{ \vr \vc{F}_{k}(\vr, \vu) \cdot \left( \vu 
- \mathcal{B} \left[ \vr - \frac{M}{|Q|}  \right]  \right) } \right) \D W_{k}
\end{split}
\end{equation}
for all $\Phi \in C^2$, $\Phi' \geq 0$. We clearly have
\begin{align*}
\mathcal D&\leq\,\intQ{\big(\varrho|\bfu|^2+p(\varrho)+1\big)} \lesssim\intQ{ \Big[ \mathbb{S}(\Grad \vu) : \Grad \vu + \varepsilon\delta p(\vr) +1\Big]}
\end{align*}
using \eqref{i10} as well as
\begin{align}\label{eq:611}
\begin{aligned}
\sum_{k\geq 1} &\left( \intQ{ \vr \vc{F}_k(\vr, \vu) \cdot \left( \vu 
- \mathcal{B} \left[ \vr - \frac{M}{|Q|}  \right] \right) } \right)^2 \\&\lesssim \sum_{k\geq 1}f_{k}^{2}\bigg(\intQ{ \varrho (1+|\bfu|^{\alpha+1})}\bigg)^2
\leq \kappa\mathcal D^2+c_{\kappa}
\end{aligned}
\end{align}
for an arbitrary $\kappa\in (0,1)$ using \eqref{p12}, \eqref{i7} and continuity of $\mathcal B$.
Plugging these estimates into \eqref{m1a} and applying expectations yields
\begin{equation} \label{m1}
\begin{split}
\Big[ \mathbb E& [|\mathcal{D}|^m ]\Big]_{t = \tau_1- }^{t = \tau_2+}+ D_{m} \int_{\tau_1}^{\tau_2} \mathbb E[|\mathcal{D}|^m]  \dt\leq  C_m(\tau_2-\tau_1)
\end{split}
\end{equation}
for all integers $m > 0$ with some positive constants $C_m,D_m$. Here the passage from $\tau_{2}-$ to $\tau_{2}+$ follows from the fact that $\mathcal{D}$ is a sum of a non-increasing function and a continuous one, specifically, $\mathcal{D}(\tau_{2}+)\leq \mathcal{D}(\tau_{2}-)$. Also note that we approximated the mapping $\mathcal D\mapsto |\mathcal D|^m$ by a sequence of smooth functions $\Phi$ with bounded derivatives.

Thus
Lemma \ref{uL1} together with  \eqref{m1} give rise to the uniform bound
\[
\expe{ |\mathcal{D}|^m (\tau ) } \leq \exp( - D_m \tau) \left( \expe{ |\mathcal{D}(0)|^m} - \frac{C_m}{D_m} \right) + \frac{C_m}{D_m}  \quad \mbox{for all}\quad \tau > 0. 
\]
In view of the bound
\begin{equation}\label{eq:62}
|\mathcal{D}(\tau)-\mathcal{E}(\tau)|\lesssim \sqrt{\mathcal{E}(\tau)}
\end{equation}
which holds true for all $\tau\geq 0$, we deduce that \eqref{eq:61} follows.
Consequently, the first claim of Theorem \ref{thm:bound} holds and the proof is complete.

\medskip

As  a consequence, we may also control the supremum over time inside expectation.

\begin{Corollary}\label{c:3.2}
Let $T>0$. Under the assumptions of Theorem~\ref{thm:bound} it holds
\begin{equation*}
\begin{split}
\expe{ \sup_{\tau \in [0,T]} \mathcal{E}^m (\tau) } &+ \expe{ \int_{0}^T \mathcal{E}^{m-1} \intQ{ \left[ \mathbb{S} (\Grad \vu) : \Grad \vu + p(\vr) \right] } 
\dt }\\
& \lesssim \expe{ \mathcal{E}^{m}(0)} +c_{T},
\end{split}
\end{equation*}
where the function  $T\mapsto c_{T}>0$ is locally bounded on $[0,\infty)$.
\end{Corollary}

\begin{proof}
We consider \eqref{m1a} and take first supremum over time and then expectation. In order to estimate the stochastic integral, we apply Burkholder--Davis--Gundy's inequality,  \eqref{eq:611} and Young's inequality to obtain
\begin{align*}
&\expe{\sup_{\tau\in[0,T]}\left|\sum_{k\geq 1}\int_{0}^{\tau}  |\mathcal{D}|^{m-1} \left( \intQ{ \vr \vc{F}_{k}(\vr, \vu) \cdot \left( \vu 
- \mathcal{B} \left[ \vr - \frac{M}{|Q|}  \right]  \right) } \right) \D W_{k}\right|}
\end{align*}
\[ \leq \mathbb{E} \left[ \left( \int_{0}^T (2\kappa | \mathcal{D} |^{2
   m} + c_{\kappa}) \dt\right)^{\frac{1}{2}} \right] 
\leq \kappa \mathbb{E} \left[\sup_{\tau \in [0, T]} | \mathcal{D} |^m(\tau)\right] +
    \mathbb{E} \left[ \int_{0}^T | \mathcal{D} |^m \dt \right] +
   c_{\kappa, T}. \]
The first term on the above right hand side can be absorbed into the left hand side of the estimate. The second term on the  right hand side is controlled in view of \eqref{m1} by the initial value. Altogether, we deduce
$$
\expe{\sup_{\tau\in[0,T]}|\mathcal{D}|^{m}(\tau)}+\expe{\int_{0}^{T}|\mathcal{D}|^{m-1}\intQ{\left[ \mathbb{S} (\Grad \vu) : \Grad \vu + p(\vr) \right]}\dt}
$$
$$
\lesssim \expe{ |\mathcal{D}|^{m}(0)} +c_{T},
$$
which yields the claim by using \eqref{eq:62}.
\end{proof}

\section{Asymptotic compactness}
\label{co}

Our next goal is to prove Theorem \ref{thm:comp}. 

\subsection{Global energy estimate}\label{s:gle}

First fix a time interval $[-T,T]$. In view of the uniform bounds established in 
Theorem \ref{thm:bound} and Corollary~\ref{c:3.2}, we claim that
\begin{equation} \label{co1}
\begin{aligned}
\expe{ \sup_{\tau \in [-T,T]} \mathcal{E}^m_n (\tau) } &+ \expe{ \int_{-T}^T \mathcal{E}^{m-1}_n \intQ{ \left[ \mathbb{S} (\Grad \vu_n) : \Grad \vu_n + p(\vr_n) \right] } 
\dt }\\
&\qquad \lesssim \mathcal{E}_\infty (m)+c_{2T}
\end{aligned}
\end{equation}
$m=1, \dots, 4$, 
for all $n=1,2,\dots.$ Indeed, from Corollary~\ref{c:3.2} applied to the dissipative martingale solution $[\vr_{n},\vu_{n},W_{n}]$ on the time interval $[-T,T]$ such that $T_{n}>T$, we obtain a bound of the left hand side in \eqref{co1} of the form
$$
\lesssim\expe{ \mathcal{E}^{m}_{n}(-T)}+c_{2T},
$$
where the implicit constant as well as $c_{2T}>0$ is universal, i.e. independent of the solution $[\vr_{n},\vu_{n},W_{n}]$. Recall that   $[\vr_{n},\vu_{n},W_{n}]$ solves the system on $[-T_{n},\infty)$ and that $T_{n}\to\infty$. Hence, employing Theorem~\ref{thm:bound}, in particular \eqref{eq:61}, and the uniform bound \eqref{eq:ass}, we conclude that if $n$ is sufficiently large so that $-T+T_{n}> M$ for some $M>0$  large enough, then
$$
\expe{ \mathcal{E}^{m}_{n}(-T)}\lesssim \mathcal{E}_{\infty}(m),
$$
which proves \eqref{co1} for $n\geq n(T)$ large enough.

On the other hand, for $n< n(T)$ and $T_{n}>T$ the left hand side of \eqref{co1} is bounded using Corollary~\ref{c:3.2} by
\begin{equation*}
\begin{aligned}
\lesssim  \expe{ \mathcal{E}_{n}^{m}(-T_{n})} +c_{-T+T_{n}}\lesssim 1,
\end{aligned}
\end{equation*}
where the last inequality follows from the fact that  $T\mapsto c_{T}$ is locally bounded on $[0,\infty)$ and $-T+T_{n}\leq M$.

The bound \eqref{co1} for the case $T_{n}<T$ follows from \eqref{eq:ass} and Remark~\ref{r:2.4}.

\subsection{Pressure estimates}
\label{sec:pressureestimates}

Given \eqref{co1} we show that the pressure $p(\vr_n)$ is bounded in a reflexive space $L^r((-T,T) \times Q)$ for some 
$r > 1$. To see this, we use the identity \eqref{i5} with 
\[
\beta (\vr_n) = (\Ov{\vr} - \vr_n )^{-\omega}, \ \omega > 0 \ \mbox{sufficiently small},
\]
obtaining $\prst_{n}$-a.s.
\begin{align} \label{co2}
\begin{aligned}
&\int_{-T}^{T} \intQ{ p(\vr_n) (\Ov{\vr} - \vr_n )^{-\omega} } \dt\\ &=
\left[ \intQ{ \vr_n \vu_n \cdot \mathcal{B} \left[ (\Ov{\vr} - \vr_n )^{-\omega} - \frac{1}{|Q|} \intQ{ (\Ov{\vr} - \vr_n )^{-\omega} } \right] }  \right]_{t= -T}^{t = T}
\\&+ \frac{1}{|Q|} \int_{-T}^{T} \left( \intQ{ p(\vr_n) } \intQ{ (\Ov{\vr} - \vr_n )^{-\omega} } \right) \dt \\
&- \int_{-T}^{T} \intQ{  \vr_n \vu_n \otimes \vu_n : \Grad \mathcal{B} \left[ (\Ov{\vr} - \vr_n )^{-\omega} - \frac{1}{|Q|} \intQ{ (\Ov{\vr} - \vr_n )^{-\omega} }  \right] } \dt\\
&+ \int_{-T}^{T} \intQ{  \mathbb{S} ( \Grad \vu_n ) : \Grad \mathcal{B} \left[(\Ov{\vr} - \vr_n )^{-\omega} - \frac{1}{|Q|} \intQ{ (\Ov{\vr} - \vr_n )^{-\omega} }  \right] } \dt \\
& - \int_{-T}^{T} \intQ{ \vr_n \vu_n \cdot \mathcal{B} \Big[ \Div ( (\Ov{\vr} - \vr_n )^{-\omega} \vu_n ) \Big] } \dt
\\
&- \sum_{k\geq 1}\int_{-T}^{T} \left( \intQ{  \vr_n \vc{F}_{k}(\vr_n, \vu_n) \cdot \mathcal{B} \left[ (\Ov{\vr} - \vr_n )^{-\omega} - \frac{1}{|Q|} \int_Q(\Ov{\vr} - \vr_n )^{-\omega}  \dx  \right] } \right) \D W_{k} \\
&+ I_n,
\end{aligned}
\end{align}
where we have set
\[
\begin{split}
&I_n =  \int_{-T}^{T} \intQ{ \vr_n \vu_n 
\\
&\qquad\cdot \mathcal{B} \left[ \left( 
\frac{    
(\Ov{\vr} - (\omega - 1) \vr_n) }{(\Ov{\vr} - \vr_n )^{\omega + 1} } \right) \Div \vu_n 
- \frac{1}{|Q|}\intQ{ \left( 
\frac{    
(\Ov{\vr} - (\omega - 1) \vr_n) }{(\Ov{\vr} - \vr_n )^{\omega + 1} } \right)\Div \vu_n }  \right] } \dt.
\end{split}
\] 

In view of hypothesis \eqref{p1}, the pressure potential satisfies 
\begin{equation} \label{co3}
(\Ov{\vr} - \vr)^{-\beta + 1} \aleq P(\vr); 
\end{equation}
whence 
\[
\| (\Ov{\vr} - \vr_n )^{-\omega} \|^s_{L^s(Q)} \aleq \intQ{ P(\vr_n ) } \leq \mathcal{E} 
\ \mbox{as long as}\ \omega s \leq (\beta - 1).  
\]
Consequently, if $\omega > 0$ is chosen small enough, all integrals on the right--hand side of \eqref{co2} except $I_n$ are controlled 
by the energy bounds \eqref{co1}.

As for $I_n$, the smoothing properties of the operator $\mathcal{B}$ specified in \eqref{def:bog} 
can be used to control $I_n$ by the energy as long as the quantity 
\[
(\Ov{\vr} - \vr_n)^{-(\omega + 1)} \Div \vu_n 
\] 
can be estimated in $L^{2}(-T,T;L^1(Q))$. In view of \eqref{co3}, this requires
\[
2 (\omega + 1) \leq (\beta - 1), 
\]
meaning $0< \omega\leq \frac{\beta-3}{2}$
which is possible as $\beta > 3$. 

Passing to expectations in \eqref{co2} we may therefore infer that
\begin{equation} \label{co4}
\expe{ \int_{-T}^T \intQ{ p(\vr_n) (\Ov{\vr} - \vr_n )^{-\omega} } \dt } \leq c \left( \mathcal{E}_\infty(m), T \right),\ 
0 < \omega \leq \frac{ \beta - 3}{2}.
\end{equation} 
Note that, in view of hypothesis \eqref{p1}, 
\[
\left\| p(\vr_n) \right\|^{ \frac{ \beta + \omega }{\beta} }_{L^ \frac{ \beta + \omega }{\beta}((-T,T) \times Q) }
\aleq \left( 1 + \int_{-T}^T \intQ{ p(\vr_n) (\Ov{\vr} - \vr_n )^{-\omega} } \dt  \right),
\]
whence
$$
\expe{ \left\| p(\vr_n) \right\|^{ \frac{ \beta + \omega }{\beta} }_{L^ \frac{ \beta + \omega }{\beta}((-T,T) \times Q) } }\leq c \left( \mathcal{E}_\infty(m), T \right).
$$

\subsection{Limit process}

The energy estimate \eqref{co1} and the pressure estimate \eqref{co4} are exactly the same as those obtained in the 
existence theory. Following the stochastic compactness arguments of \cite[Chapter 4]{BrFeHobook} or rather \cite[Chapter 7]{BrFeHobook} which also gives the necessary  additional details regarding the trajectory space $\mathcal{T}$,
we may use  Jakubowski--Skorokhod's representation theorem and find a new sequence of random variables $\tvr_n$, $\tvu_n$, 
with associated cylindrical Wiener processes $\tvW_n$ defined on the standard probability space $(\Omega,\mathfrak{F},\prst)=([0,1], \mathfrak{B}, 
\D y )$ such that (up to a subsequence)
\[
\mathcal{L}_{\mathcal{T}} [\vr_n, \vu_n, W_n ] = \mathcal{L}_{\mathcal{T}}[\tvr_n, \tvu_n, \tvW_n]  
\]
for any $1\leq q<\infty$. In addition, there exists a process $[\vr,\vu,W]$ such that
\begin{equation} \label{co5}
\begin{split}
\tvr_n &\to \vr \ \mbox{in}\ C_{\rm{weak,loc}}([-T; T]; L^q(Q)),\\
\tvu_n &\to \vu \ \mbox{in}\ \left( L^2(-T,T; W^{1,2}(Q; R^d)) , w \right)\\ 
\tvW_n &\to  W \ \mbox{in}\ C([-T,T]; \mathfrak{U}_0 )  
\end{split}
\end{equation}
for any $T > 0$ $\prst-$a.s. In particular, 
\[
\mathcal{L}_{\mathcal{T}} [\vr_n , \vu_n,  W_n ] 
= \mathcal{L}_{\mathcal{T}} [\tvr_n , \tvu_n,  \tvW_n ] \to 
\mathcal{L}_{\mathcal{T}} [\vr , \vu,  W ] \ \mbox{narrowly as}\ n \to \infty.  
\]

\subsection{Asymptotic compactness of densities}
\label{s:4.4}

To finish the proof of Theorem \ref{thm:comp}, it remains to show that $[\vr, \vu,  W]$ is an entire solution. 
This can be done similarly to \cite[Section 4.5]{BrFeHobook} as soon as we are able to show strong convergence of the 
density sequence $\{ \tvr_n \}_{n=1}^\infty$ $\prst-$a.s. This is a delicate issue as we have no information on compactness of 
``initial data''. 

First observe that \eqref{co5} yields the equation of continuity for the limit functions, namely 
\begin{equation} \label{co6}
\int_R \intQ{ \left[ \vr \partial_t \varphi + \vr \vu \cdot \Grad \varphi \right] } \dt = 0 
\end{equation}
for any $\varphi \in C^1_c(R \times \Ov{Q})$ $\prst$-a.s. Moreover, as $0 \leq \vr \leq \Ov{\vr}$ and $\vu 
\in L^2_{\rm loc}(R; W^{1,2}_0(Q; R^d)$, we may use the standard regularization technique of DiPerna and Lions \cite{DL} 
to deduce the renormalized version of \eqref{co6}, 
\begin{equation} \label{co7}
\int_R \intQ{ \left[ \vr \log(\vr) \partial_t \varphi + \vr \log(\vr) \vu \cdot \Grad \varphi - 
\vr \Div \vu \varphi \right] } \dt = 0 
\end{equation}
for any $\varphi \in C^1_c(R \times \Ov{Q})$ $\prst$-a.s.

Next, given $\varphi \in C^1_c(R \times \Ov{Q})$, we also have 
\begin{equation} \label{co8}
\int_R \intQ{ \left[ \tvr_n \log(\tvr_n) \partial_t \varphi + \tvr_n \log(\tvr_n) \tvu_n \cdot \Grad \varphi - 
\tvr_n \Div \tvu_n \varphi \right] } \dt = 0 
\end{equation}
for any $\varphi \in C^1_c(R \times \Ov{Q})$ $\prst$-a.s. To be able to let $n \to \infty$ in \eqref{co8}, we must extend 
the convergence stated in \eqref{co5} to nonlinear functions of $(\vr, \vu, \Grad \vu)$. This is possible as the Skorokhod 
argument can be extended to any composition (cf. \cite[Proposition 4.5.5]{BrFeHobook})
\[
B(\vr, p(\vr), \vu, \Grad \vu) 
\]
as long as
\begin{equation} \label{co9}
\expe{ \int_{-T}^T \intQ{ | B(\vr_n, p(\vr_n), \vu_n, \Grad \vu_n) |^r } } \dt \leq c (T ) 
\end{equation}
uniformly for
$n \to \infty$ for some $r > 1$.
Consequently, we may assume, in addition to \eqref{co5} that 
\begin{equation} \label{co10}
B(\tvr_n, p(\tvr_n), \tvu_n, \Grad \tvu_n) \to 
\Ov{B(\vr, p(\vr), \vu, \Grad \vu) } \ \mbox{weakly in}\ L^r ((-T,T) \times Q) 
\end{equation}
for any $T > 0$ $\prst$-a.s. as soon as \eqref{co9} holds. In particular, we may let $n \to \infty$ in \eqref{co8} to obtain 
\begin{equation} \label{co11a}
\int_R \intQ{ \left[ \Ov{\vr \log(\vr) } \partial_t \varphi + \Ov{ \vr \log(\vr)  \vu } \cdot \Grad \varphi - 
\Ov{ \vr \Div \vu } \varphi \right] } \dt = 0 
\end{equation}
for any $\varphi \in C^1_c(R \times \Ov{Q})$ $\prst$-a.s. 

Subtracting \eqref{co7} from \eqref{co11a} and using spatially homogeneous test functions $\varphi$ yields an ODE for the 
oscillation defect 
\[
D(t) = \intQ{ \left[ \Ov{\vr \log(\vr) } - \vr \log(\vr) \right] (t, x) },
\]
namely 
\begin{equation} \label{co12}
\frac{\D }{\dt} D(t) + \intQ{ \left[ \Ov{\vr \Div \vu } - \vr \Div \vu \right] (t,x) } = 0 \ \mbox{for a.a.}\ t \in R.
\end{equation}

The existence theory for the compressible Navier--Stokes system leans on \emph{Lions identity}
\begin{equation} \label{co13}
\intQ{ \left[ \Ov{\vr \Div \vu } - \vr \Div \vu \right] (t,x) } = 
\intQ{ \left[ \Ov{ p(\vr) \vr } - \Ov{p(\vr)} \vr \right] (t,x) } \ \mbox{for a.a.}\ t,
\end{equation}
see Lions \cite{LI4}. Validity of \eqref{co13} has been extended to sequences of (approximate) solutions in \cite[Section 4.5]{BrFeHobook}. Plugging \eqref{co13} in \eqref{co12} yields $\prst$-a.s. 
\begin{equation} \label{co14}
\frac{\D }{\dt} D(t) + \intQ{ \left[ \Ov{ p(\vr) \vr } - \Ov{p(\vr)} \vr \right] (t,x) } = 0 \ \mbox{for a.a.}\ t \in R.
\end{equation}
As $\vr \mapsto p(\vr)$ is non--decreasing, we have
\[
\intQ{ \left[ \Ov{ p(\vr) \vr } - \Ov{p(\vr)} \vr \right] (t,x) } \geq 0; 
\]
whence the defect $D$ is a non--increasing function of time. This immediately yields the desired conclusion 
$D \equiv 0$ as soon as we know that $D(t_0) = 0$ for some $t_0$, which is for instance the case for solutions of the initial--value problem 
emanating from a compact sequence of initial data. 

In our situation, we need to proceed differently. We apply pathwise the deterministic argument borrowed from \cite{EF53}.  
On the one hand, as $\vr \in [0, \Ov{\vr}] \mapsto \vr \log (\vr)$ is $\alpha-$H\" older continuous for any $0 < \alpha < 1$, we have 
\[
|\vr_n \log(\vr_n) - \vr \log(\vr) | \aleq | \vr_n - \vr |^\alpha; 
\]
whence, by H\" older's inequality, 
\[
\intQ{ \left[\Ov{\vr \log(\vr)} - \vr \log(\vr)\right] } 
\aleq \lim_{n \to \infty} \intQ{ | \vr_n - \vr |^{\alpha } } \aleq 
\left( \lim_{n \to \infty}\intQ{ | \vr_n - \vr |^{ \gamma + 1 } } \right)^{\frac{\gamma + 1}{\alpha}}. 
\]
On the other hand, as the pressure satisfies \eqref{p1}, 
\[
\intQ{ \left[\Ov{p(\vr) \vr} - \Ov{p(\vr)} \vr\right] } \geq a \intQ{\left[\Ov{\vr^\gamma \vr} - \Ov{\vr^\gamma } \vr\right] } 
\geq a \lim_{n \to \infty} \intQ{| \vr_n - \vr |^{\gamma + 1} }.
\]
Consequently, we deduce from \eqref{co14} 
\begin{equation} \label{co15}
\frac{\D }{\dt} D(t) + \theta D(t)^{\frac{\gamma + 1}{\alpha}} \leq 0 \ \mbox{for a.a.}\ t \in R
\end{equation}
for some $\theta > 0$. Since $0 \leq D \leq \Ov{D}$ for any $t \in R$, we obtain the desired conclusion $D = 0$ yielding 
strong $L^1-$convergence of $\{ \tvr_n \}_{n=1}^\infty$ $\prst$-a.s. 

We have proved Theorem \ref{thm:comp}.

\section{Construction of stationary solutions}
\label{cs}

The goal of this section is to prove Theorem~\ref{thm:main}. Therefore, let $[\varrho, \tmmathbf{u}, W]$ be a dissipative martingale solution on $[0,\infty)$ defined on some stochastic basis $(\Omega,\mathfrak{F},(\mathfrak{F}_{t})_{t\geq0},\prst)$ and satisfying \eqref{hypo1}.
We define the probability measures
\begin{equation}\label{eq:6}
\nu_S \equiv \frac{1}{S} \int_0^S \mathcal{L}_{\mathcal{T}} (S_t [\varrho, \tmmathbf{u},
   W]) \dt \in \mathfrak{P} (\mathcal{T}) .
   \end{equation}
More precisely, the time average  is defined as a narrow limit of  Riemann sums, i.e. for every $F\in BC(\mathcal{T})$ we have for a sequence of equidistant partitions $\{0=t_{0}<t_{1}<\cdots<t_{N}=S\}$
 $$
 \left[\frac{1}{S}\int_{0}^{S}\mathcal{L}(S_{t}[\varrho, \tmmathbf{u}, W])\dt\right](F)=\lim_{N\to\infty}\left[\frac{1}{N}\sum_{i=0}^{N-1}\mathcal{L}(S_{t_{i}}[\varrho, \tmmathbf{u}, W])(F)\right].
 $$
 As explained in Remark~\ref{r:2.4}, also in \eqref{eq:6} we tacitly regard functions defined on  time intervals $[-t,\infty)$ as trajectories on $R$ by extending them to $s\leq-t$ by the value at $-t$.

The proof of Theorem~\ref{thm:main} now proceeds in two main steps. First, we prove tightness of the above measures and apply Prokhorov's theorem in order to obtain a narrowly converging subsequence. Second, in view of Theorem~\ref{thm:comp} we identify the limit measure as a law of a stationary solution.

\begin{Proposition}\label{p:5.1}
The family of measures $\{\nu_{S};\,S>0\}$
is tight in $\mathcal{T}$.
\end{Proposition}

\begin{proof}
Choose a bounded interval $[- T, T]$. In order to prove tightness of $\nu_{S}$, we first  prove tightness of the laws of the time shifts $S_{t}[\vr,\vu,W]$, $t\geq 0$. Since $[\vr,\vu,W]$ solves the system on $[0,\infty)$, its time shift $S_{t}[\vr,\vu,W]$ is a solution on $[-t,\infty)$. In view of \eqref{hypo1}, we may apply the considerations of Section~\ref{s:gle} applied to the time shifts $S_{t}[\vr,\vu,W]$ to deduce
\begin{equation} \label{co11}
\begin{aligned}
\expe{ \sup_{s \in [-T,T]} \mathcal{E}^m (s+t) } &+ \expe{ \int_{-T}^T \intQ{ |\Grad \vu|^{2}(\cdot+t) }
\,\D s } \lesssim \mathcal{E}_\infty (m)+c_{2T}
\end{aligned}
\end{equation}

 The important point is that the bound depends on the length of the time interval but not on the time shift. As a consequence, the time shifts $\vu(\cdot+t)$ are tight on $L^{2}_{\rm{loc}}(R;W^{1,2}_{0}(Q;R^{d}))$ equipped with the weak topology.
Moreover, from the continuity equation we get for all $t\geq 0$
\[ \mathbb{E} [\| \varrho (\cdummy + t) \|_{C^1 ([- T, T] ; W^{- 1, 2})}]
   \leqslant C (T) \]
This implies tightness of
$\varrho (\cdummy + t)$ on $C_{\tmop{weak}, \tmop{loc}} (R ; L^q(Q))$, $1\leq q<\infty$, using also  the boundedness of ${\varrho}$, . 

This already implies tightness of the  projection of the time averaged measures $\nu_{S}$ to the first two components. Indeed, if $\varepsilon > 0$ is given and $K_{\varepsilon}$ is the associated compact set in $C_{\tmop{weak}, \tmop{loc}} (R ; L^q(Q))$,
such that
$$
\sup_{t\geq0}\mathcal{L} (S_t \varrho) (K^c_{\varepsilon})<\varepsilon,
$$
then
\[ \frac{1}{S} \int_0^S \mathcal{L} (S_t \varrho) (K^c_{\varepsilon}) d t < \varepsilon .
\]
The argument for the projection to $\vu$ is the same.

In addition, we recall that the shift chosen on the noise is $S_t W = W
(\cdummy + t) - W (t)$. Consequently, every $S_t W$ is a Wiener process with
$S_t W (0) = 0$. That means that all $S_t W$ have the same law which is tight on
$C_{\tmop{loc}, 0} (R, \mathfrak{U}_0)$. Altoghether, the claim follows.
\end{proof}

As the next step, we observe that limits of the ergodic averages are invariant under various shifts. This in particular implies shift invariance of any accumulation point of the time averages $\nu_{S}$ as we will see below.

\begin{Lemma}\label{l:5.2}
It holds
\[ \frac{1}{S - \kappa_3} \int_{\kappa_1}^{S + \kappa_2} \mathcal{L} (S_{t +
   \tau} [\varrho, \tmmathbf{u}, W]) d t - \frac{1}{S} \int_0^S \mathcal{L}
   (S_t [\varrho, \tmmathbf{u}, W]) d t \rightarrow 0\]
narrowly
as $S \rightarrow \infty$ for all $\kappa_1, \kappa_2, \kappa_3, \tau \in R$.
\end{Lemma}

\begin{proof}
Let $G \in BC (\mathcal{T})$. Using the continuity of the time shifts $t\mapsto S_{t}$ on $\mathcal{T}$, we have $G\circ S_{t} \in BC (\mathcal{T})$ and it holds
\begin{align*}
  \frac{1}{S - \kappa_3} \int_{\kappa_1}^{S +
   \kappa_2} \mathcal{L} (S_{t + \tau} [\varrho, \tmmathbf{u}, W]) (G) \dt
   & =
   \frac{1}{S - \kappa_3} \int_{\kappa_1}^{S +
   \kappa_2} \mathcal{L} ([\varrho, \tmmathbf{u}, W]) (G \circ S_{- t -
   \tau}) \dt \\
& =\frac{1}{S - \kappa_3} \int_{\kappa_1 +
   \tau}^{S + \kappa_2 + \tau} \mathcal{L} ([\varrho, \tmmathbf{u}, W])
   (G \circ S_{- s}) \,\D s \\
& =  \frac{S}{S - \kappa_3} \frac{1}{S}
   \int_0^{S} \mathcal{L} ([\varrho, \tmmathbf{u}, W]) (G \circ S_{- s}) \,\D s\\
&  - \frac{1}{S - \kappa_3} \int_0^{\kappa_1 +
   \tau} \mathcal{L} ([\varrho, \tmmathbf{u}, W]) (G \circ S_{- s}) \,\D s \\
& +  \frac{1}{S - \kappa_3} \int_{S}^{S +
   \kappa_2 + \tau} \mathcal{L} ([\varrho, \tmmathbf{u}, W]) (G \circ S_{-
   s}) \,\D s.
   \end{align*}
Using boundedness of $G$, the above has the same narrow asymptotic limit as
\[ \frac{1}{S}
   \int_0^{S} \mathcal{L} ([\varrho, \tmmathbf{u}, W]) (G \circ S_{- s}) \,\D s= \frac{1}{S} \int_0^{S} \mathcal{L} (S_s
   [\varrho, \tmmathbf{u}, W]) (G) \,\D s, \]
   which finishes he proof.
\end{proof}

As a consequence, we observe that if the narrow limit of
\[  \nu_{\tau, S_{n}} \equiv \frac{1}{S_{n}} \int_0^{S_{n}} \mathcal{L} (S_{t + \tau}
   [\varrho, \tmmathbf{u}, W]) \dt  \]
in
$\mathfrak{P} (\mathcal{T})$ as $n \rightarrow \infty$ exists for some $\tau=\tau_0 \in R$ then it exists for all $\tau \in R$ and is independent of the choice
of $\tau$. 

In view of Proposition~\ref{p:5.1} and Lemma~\ref{l:5.2} together with Jakubowski--Skorokhod's theorem, there exists a sequence $S_{n}\to\infty$ and $\nu\in\mathfrak{P} (\mathcal{T})$ so that $\nu_{0,
S_n} \rightarrow \nu$ narrowly in $\mathfrak{P} (\mathcal{T})$ as well as $\nu_{\tau,
S_n} \rightarrow \nu$ narrowly for all $\tau \in R$. Accordingly, the limit
measure $\nu$ is shift invariant in the sense that for every $G\in BC(\mathcal{T})$ and every $\tau\in R$ we have
$$
\nu(G\circ S_{\tau})=\lim_{n\to\infty}\nu_{S_{n}}(G\circ S_{\tau})=\lim_{n\to\infty}\nu_{-\tau,S_{n}}(G)=\nu(G).
$$

To conclude the proof of Theorem~\ref{thm:main}, it remains to show that $\nu$ is a law of an entire solution to \eqref{E1}--\eqref{E3a} in the sense of Definition~\ref{def:ent}.
We begin with an auxiliary proposition.

\begin{Proposition}\label{p:5.3}
  Let $[\varrho, \tmmathbf{u}, W]$ be a dissipative martingale solution on $(-T,\infty)$ defined on some probability space $(\Omega,\mathfrak{F},\prst)$. Let $S>0$ be arbitrary. Then every process $[\tilde\varrho, \tilde{\tmmathbf{u}}, \tilde W]$ defined on any probability space and having the law
  $$
\nu_{S}\equiv \frac{1}{S}\int_{0}^{S}\mathcal{L}(S_{t}[\varrho, \tmmathbf{u}, W])\dt\in\mathfrak{P}(\mathcal{T})
  $$ 
   is a  dissipative martingale solution on $(- T, \infty)$.
\end{Proposition}

\begin{proof}
 Let $[\tilde\varrho, \tilde{\tmmathbf{u}}, \tilde W]$ be the process from the statement of the proposition defined on some probability space $(\Omega,{\mathfrak{F}},\prst)$ and having the
  law $\nu_{S}$. Note that such a process always exists on the canonical probability space $(\mathcal{T},\mathfrak{B}(\mathcal{T}),\nu_{S})$ or on $([0,1],\mathfrak{B}[0,1],\D y)$ by Jakubowski--Skorokhod's theorem.  We define $(\mathfrak{F}_{t})_{t\geq -T}$ as the joint canonical filtration of $[\tilde\varrho, \tilde{\tmmathbf{u}}, \tilde W]$. We intend to  show that $((\Omega,\mathfrak{F},({\mathfrak{F}}_{t})_{t\geq -T},\prst),\tilde\varrho, \tilde{\tmmathbf{u}}, \tilde W)$ is a dissipative martingale solution on $(-T,\infty)$.

By \cite[Lemma 2.1.35]{BrFeHobook}, $\tilde W$ is a cylindrical Wiener process  with respect to its canonical filtration and  $\tilde W(0)=0$.  As the next step, we want to strengthen this and show that $\tilde W$  is non-anticipative with respect to  the joint filtration $(\mathfrak{F}_{t})_{t\geq -T}$, which in view of \cite[Corollary~2.1.36]{BrFeHobook} implies that it is a cylindrical Wiener process with respect to $({\mathfrak{F}}_{t})_{t\geq -T}$, as required in Definition~\ref{def:sol}.
  To this end, we  observe that for
  any $F : \mathcal{T} \rightarrow R$ bounded Borel we have
 \begin{equation}\label{eq:convex}
 \nu_{S}(F)=\frac{1}{S}\int_{0}^{S}\mathcal{L}(S_{t}[\varrho, \tmmathbf{u}, W])(F)\dt.
\end{equation}
This is due to the fact that the time average is defined as a narrow limit of  Riemann sums
 and the extension to bounded Borel functions follows by the dominated convergence theorem.

 We know that for every $t\geq 0$ the joint canonical filtration generated by $S_{t}[\vr,\vu,W]$ is non-anticipative with respect to $S_{t}W$ in the sense that
  \begin{equation}\label{eq:2}
  \mathbb{E}\left[h_{1}\big(S_{t}[\vr,\vu,W]|_{(-T,s]}\big)h_{2}\big(S_{t}W(s+\tau)-S_{t}W(s)\big)\right]=0
  \end{equation}
  for every $s\geq -T$, every bounded continuous functions $h_{1}:\mathcal{T}|_{(-T,s]}\to R$  and $h_{2}:\mathfrak{U}_{0}\to R$ and every $\tau\geq0$. Therefore, the integrand in \eqref{eq:2} can be written as a composition of a bounded continuous function $F:\mathcal{T}\to R$ with $S_{t}[\vr,\vu, W]$, where $F$ does not depend on $t\geq0$. Using this function in \eqref{eq:convex} we obtain
  $$
  \mathbb{E}\left[h_{1}\big(\tilde\vr,\tilde\vu,\tilde W)|_{(-T,s]}\big)h_{2}\big(\tilde W(s+\tau)-\tilde W(s)\big)\right]= \nu_{S}(F)
  $$
  $$
 =\frac{1}{S}\int_{0}^{S}\mathbb{E}\left[h_{1}\big(S_{t}[\vr,\vu,W]|_{(-T,s]}\big)h_{2}\big(S_{t}W(s+\tau)-S_{t}W(s)\big)\right]\dt=0,
  $$
   where we slightly abused the notation: each expected value $\mathbb{E}$ possibly refers to a different probability measure as   the processes $[\tilde\varrho, \tilde{\tmmathbf{u}}, \tilde W]$ and $[\varrho, \tmmathbf{u}, W]$ can be defined on  different probability spaces. 
 Thus, $\tilde W$ is non-anticipative with respect to the joint filtration $(\mathfrak{F}_{t})_{t\geq0}$.

  Let $\psi\in C_c^{1}((-T,\infty);C^{1} (\overline{Q}))$ and define
  \[ F (\varrho, \tmmathbf{u}, W) = G \left( \int_{-T}^{\infty} \int_{Q} \varrho \partial_t
     \psi + \varrho \tmmathbf{u} \cdummy \nabla \psi \dx\dt \right), \]
  where $G:R\to [0,\infty)$ is continuous, bounded, $G > 0$ on $R \setminus \{ 0 \}$ and $G
  (0) = 0$. Plugging this into \eqref{eq:convex} we deduce that
  \begin{equation}\label{eq:1}
  \mathbb{E}\left[F(\tilde\vr,\tilde\vu,\tilde W)\right]=\frac{1}{S}\int_{0}^{S}\mathbb{E}\big[F(S_{t}[\varrho, \tmmathbf{u}, W])\big]\dt.
  \end{equation}
 Since $[\varrho, \tmmathbf{u}, W]$ is a solution on $(-T,\infty)$, it follows that $S_{t}[\varrho, \tmmathbf{u}, W]$ is a solution on $(-T-t,\infty)$ and, in particular, the continuity equation holds on $(-T,\infty)$. Consequently, the integrand on the right-hand side of \eqref{eq:1} vanishes for all $t\geq 0$. This implies that the continuity equation is also satisfied by $[\tilde\vr,\tilde\vu,\tilde W]$ on $(-T,\infty)$.
 The same argument applies to the renormalized continuity equation  \eqref{renorm}.
  
 For the momentum equation \eqref{N2} as well as for the energy inequality \eqref{N3}  we need to proceed differently since the corresponding stochastic integrals are generally not defined as functions on the space of trajectories $\mathcal{T}$. Recall that the momentum equation  \eqref{N2} is solved on $(-T,\infty)$ by $S_{t}[\vr,\vu,W]$ for all $t\geq 0$. Furthermore, as it was showed for instance in the proof \cite[Theorem 2.9.1]{BrFeHobook}, stochastic It\^o integrals of the form
 $$
 \int_{S}^{T}\mathbb{G}(\vr,\vu)\,\D W
 $$
can be written as a composition $H(\vr,\vu,W):\Omega\to R $ where $H:\mathcal{T}\to R$ is a measurable function which is universal in the sense that it depends on $S,T$ but is independent of the process $(\vr,\vu,W)$ provided $W$ is a cylindrical Wiener process with respect to some filtration and $\mathbb{G}(\vr,\vu)$ is stochastically integrable with respect to $W$. As a consequence, also for the momentum equation, there is a bounded Borel function $F:\mathcal{T}\to R$ such that $F(\vr,\vu,W)=0$ $\prst$-a.s. if and only if $[\vr,\vu,W]$ satisfies \eqref{N2}. This function can now be used in \eqref{eq:1} to deduce that $[\tilde\vr,\tilde\vu,\tilde W]$ satisfies \eqref{N2} on $(-T,\infty)$.
 
 A similar argument can be applied for the energy inequality as well. More precisely, we put all the terms in the energy inequality on the left hand side and write the left hand side as a composition $H(\vr,\vu,W)$ for some Borel function $H:\mathcal{T}\to R$. Then we define $F=G\circ H$ where $G:R\to R$ is bounded and continuous such that $G(z)=z^{+}$ for $z\leq 1$. Applying \eqref{eq:1} we finally conclude that the energy inequality \eqref{N3} is satisfied by $[\tilde\vr,\tilde\vu,\tilde W]$. The remaining points of Definition~\ref{def:sol} are immediate and hence the proof is complete.
\end{proof}

We recall that the probability measure $\nu$ was obtained as a narrow limit of the time averages $\nu_{\tau,S_{n}}$ for any $\tau\in R$ and a sequence $S_{n}\to\infty$.
Since $\nu$ is shift invariant, any  process with law $\nu$ is
stationary as required in Definition~\ref{def:stationary}. 
To conclude the proof of Theorem~\ref{thm:main}, it remains to prove that $\nu$ is a law of 
an entire solution to \eqref{E1}--\eqref{E3a} in the sense of Definition~\ref{def:ent}.

We first consider the measures $\nu_{\tau, S_n - \tau}$, $n=1,2, \dots$, and $\tau>0$. According to Lemma~\ref{l:5.2}, it follows that the narrow limit as $n\to\infty$ exists and
\[ \lim_{n \rightarrow \infty} \nu_{\tau, S_n - \tau} = \lim_{n \rightarrow
   \infty} \nu_{0, S_n} =\nu. \]
Recall that $[\varrho, \tmmathbf{u}, W]$ from the statement of Theorem~\ref{thm:main} solves the system  on $[0, \infty)$. As a consequence,  $S_\tau
[\varrho, \tmmathbf{u}, W]$ is a solution on $[- \tau, \infty)$.
Since
$$
\nu_{\tau,S_{n}-\tau}=\frac{1}{S_{n}-\tau}\int_{0}^{S_{n}-\tau}\mathcal{L}(S_{t}S_{\tau}[\vr,\vu,W])\dt,
$$
it follows from Proposition~\ref{p:5.3} that any process with law $\nu_{\tau,S_{n}-\tau}$ is a dissipative martingale solution to \eqref{E1}--\eqref{E3a} on $[-\tau,\infty)$.

We continue by a diagonal argument: Take a sequence $\tau_m \rightarrow \infty$ and consider $\nu_{\tau_m, S_n -
\tau_m}$, $m, n \in \mathbb{N}$. Denote by $d$ the
metric on $\mathfrak{P} (\mathcal{T})$ metrizing the weak convergence. For  $m \in \mathbb{N}$, find $n = n (m) \in \mathbb{N}$ so that
\[ d (\nu_{\tau_m, S_{n (m)} - \tau_m}, \nu) < \frac{1}{m} . \]
This gives the narrow convergence
\[ \nu_{\tau_m, S_{n (m)} - \tau_m} \rightarrow \nu \quad \tmop{as} \quad m \rightarrow
   \infty . \]
   Applying  Jakubowski--Skorokhod's theorem, we obtain a sequence of approximate processes $[\tilde{\varrho}_m, \tilde{\tmmathbf{u}}_m, \tilde{W}_m]$ converging a.s. to a process $[\tilde{\varrho}, \tilde{\tmmathbf{u}}, \tilde{W}] $ in the topology of $\mathcal{T}$. Moreover, the law of $[\tilde{\varrho}_m, \tilde{\tmmathbf{u}}_m, \tilde{W}_m]$ is $\nu_{\tau_{m},S_{n(m)}-\tau_{m}}$ and necessarily the law of $[\tilde{\varrho}, \tilde{\tmmathbf{u}}, \tilde{W}] $ is $\nu$.

Finally, we observe that $[\tilde{\varrho}_m,
\tilde{\tmmathbf{u}}_m, \tilde{W}_m]$ solving the equation on $[- \tau_m,
\infty)$  satisfies the assumptions of Theorem~\ref{thm:comp}. In particular, we shall verify that \eqref{eq:ass} holds at times $-\tau_{m}$.  Recalling \eqref{eq:61} we observe that for $S_{t + \tau_m} [\varrho,
\tmmathbf{u}, W]$ we have for all $s > - t - \tau_m$
\[ \mathbb{E} \left[ \left( \int_Q E (\varrho, \varrho \tmmathbf{u}) (s + t +
   \tau_m) \,\D x\right)^m \right] \lesssim \mathbb{E} \left[ \left( \int_Q E (\varrho, \varrho \tmmathbf{u}) (0) \,\D x\right)^m \right] + c.
    \]
Hence in particular the $m$th moment of the energy of $S_{t + \tau_m}
[\varrho, \tmmathbf{u}, W]$ at time $s = - \tau_m$ is controlled this way. Since the law of $[\tilde{\varrho}_m, \tilde{\tmmathbf{u}}_m, \tilde{W}_m]$ is given by
$ \nu_{\tau_m, S_{n(m)} - \tau_m}$,
it follows
\begin{align*}
\mathbb{E} \left[ \left( \int_Q E (\tilde{\varrho}_m, \tilde{\varrho}_m
   \tilde{\tmmathbf{u}}_m) (- \tau_m) \right)^m \right] &= \frac{1}{S_{n(m)} -
   \tau_m} \int_0^{S_{n(m)} - \tau_m} \mathbb{E} \left[ \left( \int_Q E (\varrho,
   \varrho \tmmathbf{u}) (t) \right)^m \right] \dt\\
  & \lesssim  \mathbb{E} \left[ \left( \int_Q E (\varrho, \varrho \tmmathbf{u}) (0) \,\D x\right)^m \right] + c, 
  \end{align*}
which yields {\eqref{eq:ass}}. More precisely, this follows by applying \eqref{eq:convex} to a bounded truncation of the energy and then passing  to the limit.

Thus, by Theorem~\ref{thm:comp}, $[\tilde{\varrho}, \tilde{\tmmathbf{u}}, \tilde{W}] $ is an entire solution to \eqref{E1}--\eqref{E3a}. This completes the
proof of Theorem~\ref{thm:main}.

\section{Ergodic structure}\label{s:erg}

In this section we study the ergodic structure of the system \eqref{E1}--\eqref{E3a}. In particular, we show that each dissipative martingale solution as in Theorem~\ref{thm:main} gives raise to an ergodic stationary solution on the closure of its limit set. This is the best we can say at the moment, as we generally do not expect stationary solutions to  \eqref{E1}--\eqref{E3a} to be unique. As a matter of fact, already the deterministic counterpart of \eqref{E1}--\eqref{E3a} may admit infinitely many equilibrium states for a given total mass.

For a dissipative martingale solution $[\varrho, \vu,W]$ satisfying \eqref{hypo1}, we define the $\omega$--limit set as a subset of $ \mathfrak{P}
   (\mathcal{T})$ given by
\[
\begin{aligned} 
\Xi [\varrho, \vu,W] &= \big\{ \mathcal{L}_{\mathcal{T}}[r, \vc{w},B];\, 
\mbox{there exists}\ T_{n}\to\infty  \ \mbox{so that}\\
&\qquad\qquad\qquad\qquad\, S_{T_{n}}[\varrho, \vu,W]
    \rightarrow [r, \vc{w},B]\,  \mbox{in law in} \, \mathcal{T} \big\} . 
\end{aligned}
\]
In addition, according to Theorem~\ref{thm:bound} and Theorem~\ref{thm:comp}, the $\omega-$limit set $\Xi[\vr,\vu,W]$ is a non-empty set of laws of globally bounded entire solutions, which is shift  invariant and compact. Moreover, we observe that  similarly to the proof of Proposition~\ref{p:5.3} (see also \cite[Theorem 2.9.1]{BrFeHobook}), it holds that every process $[r,\vc{w},B]$ having the law of an entire solution  is an entire solution itself.
   
Let $\overline{\rm{co}} (  \Xi [\varrho, \vu,W])$ denote the closure of the convex hull of $\Xi [\varrho, \vu,W]$  with respect to the narrow convergence of probability measures.
Theorem~\ref{thm:main} then implies the following.

\begin{Corollary}\label{c:6.1}
 For every dissipative martingale solution $[\varrho, \vu,W]$ as in Theorem~\ref{thm:main}, there exists a stationary solution whose law
  belongs to $\overline{\rm{co}} (  \Xi [\varrho, \vu,W])$.
  \end{Corollary}

\begin{proof}
As discussed above, the $\omega-$limit set $\Xi[\vr,\vu,W]$ consists of laws of globally bounded entire solutions. Let $[r,\vc{w},B]$ be a process  whose law belongs to $\Xi[\vr,\vu,W]$. Then $[r,\vc{w},B]$ is itself a globally defined entire solution. Applying  the construction of a stationary solution from Section~\ref{cs} to $[r,\vc{w},B]$ instead of $[\vr,\vu,W]$, we obtain a shift-invariant measure given by a narrow limit of the form
$$
\nu=\lim_{S_{n}\to\infty}\frac{1}{S_{n}}\int_{0}^{S_{n}}\mathcal{L}_{\mathcal{T}}(S_{t}[r,\vc{w},B])\dt=\lim_{S_{n}\to\infty}\lim_{N\to\infty}\frac{1}{N}\sum_{i=0}^{N}\mathcal{L}_{\mathcal{T}}(S_{t_{i}}[r,\vc{w},B]),
$$
where the second equality is a Riemann sum approximation for an equidistant partition $\{0=t_{0}<\cdots<t_{N}=S_{n}\}$. Consequently, $\nu$ belongs to the closure of the convex hull of the laws $\mathcal{L}_{\mathcal{T}}(S_{t}[r,\vc{w},B])$, $t\geq0$, which all belong to $\Xi[\vr,\vu,W]$.
\end{proof}

Our final result shows that for every dissipative martingale solution satisfying \eqref{hypo1}, there is an associated ergodic stationary solution.

\begin{Definition} [Ergodic stationary statistical solution] \label{AD1}
A stationary statistical solution $[\vr, \vm,W]$, or its law $\mathcal{L}[\vr, \vm,W]$ on $\mathcal{T}$, is called {ergodic}, if the $\sigma$-field 
of shift invariant sets is trivial, specifically, 
\[
\mathcal{L}_{\mathcal{T}}[\vr, \vm,W] (B) = 1 \ \mbox{or}  \ \mathcal{L}_{\mathcal{T}}[\vr, \vm,W](B) = 0 \ \mbox{for any shift invariant Borel set}\ B \in \mathfrak{B}[\mathcal{T}].
\]
\end{Definition}

\begin{Theorem}\label{t:6.3}
For every dissipative martingale solution $[\varrho, \vu,W]$ as in Theorem~\ref{thm:main}, there exists an ergodic stationary solution whose law belongs $\overline{\rm{co}} (  \Xi [\varrho, \vu,W])$.
\end{Theorem}

\begin{proof}
Consider the set $\mathcal{A}$ of all stationary solutions whose law belongs to  $\overline{\rm{co}} (  \Xi [\varrho, \vu,W])$. Note that not all probability measures in $\overline{\rm{co}} (  \Xi [\varrho, \vu,W])$ are stationary, but by Corollary~\ref{c:6.1}, such a stationary solution exists, i.e. $\mathcal{A}$ is non-empty. 
Since  a convex combination of laws of entire solutions is an entire solution by the approach of Proposition~\ref{p:5.3}, $\mathcal{A}$ is convex. Due to the uniform boundedness from Theorem~\ref{thm:bound},  Theorem~\ref{thm:comp} implies that $\mathcal{A}$ is tight and closed.
Thus, by Krein--Milman's theorem, there is an extremal point of $\mathcal{A}$, which is the law of a stationary solution. Then by a classical contradiction argument (see e.g. page 30 in \cite{DPZ96}) it can be proved that this law  is ergodic.
  \end{proof}

\appendix

\section{Existence of solutions to the initial value problem}

We have the following existence result.

\begin{Theorem} \label{thm:existence}

Let $k>\tfrac{N}{2}$ and
let $\Lambda $ be a Borel probability measure defined on the space $W^{-k,2}(\Q) \times W^{-k,2}(\Q,R^d)$ such that
\begin{align*}
\Lambda&\big\{L^1(\Q) \times L^1(\Q,R^d)\big\}=1,\
\Lambda \{ \vr \geq 0 \} = 1, \\ &\quad\Lambda \bigg\{ 0 < \vr_{\mathrm{min}} \leq \intQ{ \vr } \leq \vr_{\mathrm{max}} < \infty \bigg\} = 1,
\end{align*}
for some deterministic constants $\vr_{\mathrm{min}}$, $\vr_{\mathrm{max}}$, and
\[
\int_{L^1_x \times L^1_x}  \left|\, \intQ{ \left[ \frac{1}{2} \frac{|\vc{q} |^2}{\vr} + P(\vr) \right] }
\right|^{r_0}  \D \Lambda \leq c
\]
for some $r_0\geq4$.
Let the diffusion coefficients $\mathbb{F} = (\mathbf{F}_{k})_{k\in\N}$ be continuously differentiable
satisfying \eqref{p12} and suppose that $\bfg$ is continuous satisfying \eqref{p12bis}.
Then there is a dissipative martingale solution to \eqref{E1}--\eqref{E3} in the sense of Definition \ref{def:sol} with $\Lambda=\mathcal{L}[\vr(0),\vr\bfu(0)]$. The solution satisfies uniformly in time
\begin{equation} \label{i2}
 \vr_{\mathrm{min}}  \leq \int_{\mt}\varrho(t,\cdot)\dx \leq \vr_{\mathrm{max}}  \quad \mathbb P\mbox{-a.s.}
\end{equation}
\end{Theorem}

\begin{proof}
We follow \cite[Sec. 3]{FZ} and consider for $\alpha>0$ small the approximate pressure $p_\alpha$ given by
\begin{align*}
p_\alpha(\varrho)=\begin{cases}p(\varrho)\qquad, \,\,\,\varrho\leq \overline\varrho-\alpha\\
p(\overline\varrho-\alpha)+\big([\varrho-\overline\varrho-1]^+\big)^\gamma,\,\,\, \varrho\geq \overline\varrho-\alpha
\end{cases}
\end{align*}
where $\gamma>3$.
The existence of a dissipative martingale solution\footnote{Without loss of generality we can assume that the probability space and the Wiener process do not depend on $\alpha$.}
 $$\big((\Omega,\mf,(\mf_t)_{t\geq0},\prst_\alpha),\varrho_\alpha,\bfu_\alpha,W)$$
 to  \eqref{E1}--\eqref{E2} in the sense of Definition \ref{def:sol} with the pressure $p_\alpha$
follows from \cite[Thm. 4.0.2.]{BrFeHobook} (see also \cite[Thm. 2.4]{BrHo}). Although only a pressure of the form $a\varrho^\gamma$ with $a>0$ is treated in \cite{BrFeHobook}, it is clear that the same arguments apply for any monotone pressure function which behaves asymptotically as $\varrho^\gamma$.
Also note that \cite{BrFeHobook} deals only with periodic boundary conditions. However, as demonstrated in \cite{BF} for the full Navier--Stokes--Fourier system the approach also applies to the case of bounded domains with Dirichlet boundary conditions for the velocity field (see also \cite{Smith}).\\
From the energy inequality \eqref{N3} we obtain
\begin{align}\label{apriori1}
& \mathbb E\bigg[\sup_{0\leq t\leq T}\left[ \intQ{ \left( \frac{1}{2} \vr_\alpha |\vu_\alpha|^2 + P_\alpha(\vr_\alpha) \right) } \right]^n\bigg]
\\&+ \mathbb E\bigg[\int_0^T \left( \left[ \intQ{ \left( \frac{1}{2} \vr_\alpha |\vu_\alpha|^2 + P_\alpha(\vr_\alpha) \right) } \right]^{ n - 1}\nonumber
\intQ{ \mathbb{S}(\Grad \vu_\alpha) : \Grad \vu_\alpha } \right) \,\dd\tau\bigg]\\&\leq\,c(n,T)\nonumber
\end{align}
uniformly in $\alpha$ for all $T>0$ and all $n=1,\dots, r_0$. Next we aim at establishing uniform bounds for the pressure. As in \eqref{i5} we obtain for any $T>0$
\begin{align}
\begin{aligned}
\int_{0}^{T} &\intQ{ p_\alpha(\vr_\alpha) \beta(\vr_\alpha) } \dt =
\left[ \intQ{ \vr_\alpha \vu_\alpha \cdot \mathcal{B} \left[ \beta(\vr) - \frac{1}{|Q|} \intQ{ \beta(\vr_\alpha) } \right] }  \right]_{t= 0}^{t = T}\\&+ \frac{1}{|Q|} \int_{\tau_1}^{\tau_2} \left( \intQ{ p_\alpha(\vr_\alpha) } \intQ{ \beta(\vr_\alpha) } \right) \dt \\
&- \int_{0}^{T} \intQ{  \vr_\alpha \vu_\alpha \otimes \vu_\alpha : \Grad \mathcal{B} \left[ \beta(\vr_\alpha) - \frac{1}{|Q|} \intQ{ \beta(\vr_\alpha) }  \right] } \dt
\\
&+ \int_{0}^{T} \intQ{  \mathbb{S} ( \Grad \vu_\alpha ) : \Grad \mathcal{B} \left[\beta(\vr_\alpha) - \frac{1}{|Q|} \intQ{ \beta(\vr_\alpha) }  \right] } \dt \\
& - \int_{0}^{T} \intQ{ \vr_\alpha \vu_\alpha \cdot \mathcal{B} \Big[ \Div (\beta(\vr_\alpha) \vu_\alpha) \Big] } \dt
\\
&+ \int_{0}^{T} \intQ{ \vr_\alpha \vu_\alpha \cdot \mathcal{B} \left[ \left( \beta(\vr_\alpha) - \beta'(\vr_\alpha) \vr_\alpha \right) \Div \vu_\alpha 
- A_\alpha  \right] } \dt
\\
&- \int_{0}^{T} \left( \intQ{  \vr_\alpha \vc{g}(\vr_\alpha, \vu_\alpha) \cdot \mathcal{B} \left[ \beta(\vr_\alpha) - \frac{1}{|Q|} \intQ{ \beta(\vr_\alpha) }  \right] } \right) \dt\\
&- \int_{0}^{T} \left( \intQ{  \vr_\alpha \vc{F}(\vr_\alpha, \vu_\alpha) \cdot \mathcal{B} \left[ \beta(\vr_\alpha) - \frac{1}{|Q|} \intQ{ \beta(\vr_\alpha) }  \right] } \right) \D W
\end{aligned}
 \label{i5A}
\end{align}
with $A_\alpha=\frac{1}{|Q|}\int_Q \left( \beta(\vr_\alpha) - \beta'(\vr_\alpha) \vr_\alpha \right) \Div \vu_\alpha \dx$.
After taking expectations the stochastic integral vanishes whereas all the other terms can be estimated as in the deterministic case based on \eqref{apriori1} for the choice $\beta(\varrho)=(\overline\vr-\vr)^{-\omega}$ with $\omega>0$ small. Recalling the arguments from Section \ref{sec:pressureestimates}, in particular \eqref{co3}, we can bound all terms on the right-hand side and obtain consequently
\begin{align}\label{apriori2}
\mathbb E\left\| p(\vr_n) \right\|^{ \frac{ \beta + \omega }{\beta} }_{L^ \frac{ \beta + \omega }{\beta}((0,T) \times Q) }
\aleq \left( 1 + \mathbb E\int_{0}^T \intQ{ p(\vr_n) (\Ov{\vr} - \vr_n )^{-\omega} } \dt  \right)\leq\,c(T)
\end{align}
for all $T>0$ using also \eqref{apriori1} for $n=1$.

 With estimates \eqref{apriori1} and \eqref{apriori2} at hand one can apply the stochastic compactness method based
 on the Jakubowski--Skorokhod representation theorem exactly as in
 \cite[Chapter 4.4]{BrFeHobook}. Also, we can pass
 to the limit in all terms in the equations and the energy inequality apart from the pressure 
arguing as in \cite[Chapter 4.4]{BrFeHobook}. Note that, since \eqref{apriori2} implies higher integrability of the pressure, the method from \cite{Li2} applies directly as explained in \cite[Sec. 3.6]{FZ}. 
This is in fact a purely deterministic argument and the only difference to \cite[Chapter 4.4]{BrFeHobook} is that we need to localise the effective viscous flux identity. We conclude that
$\tilde p=p(\tilde\varrho)$ which finishes the proof.
\end{proof}

\begin{Remark}\label{rem}
As can be seen from the proof, the solution constructed in Theorem \ref{thm:existence} satisfies
\begin{align*}
\mathbb E\int_0^T\int_{Q}|p(\varrho)|^{\frac{ \beta + \omega }{\beta}}\dxt\leq \,c(T)
\end{align*}
for all $T>0$ with some $\beta>0$.
\end{Remark}

\bibliography{citace}

\def\cprime{$'$} \def\ocirc#1{\ifmmode\setbox0=\hbox{$#1$}\dimen0=\ht0
  \advance\dimen0 by1pt\rlap{\hbox to\wd0{\hss\raise\dimen0
  \hbox{\hskip.2em$\scriptscriptstyle\circ$}\hss}}#1\else {\accent"17 #1}\fi}
\begin{thebibliography}{10}

\bibitem{BrFeHobook}
D.~Breit, E.~Feireisl, and M.~Hofmanov{\'a}.
\newblock {\em Stochastically forced compressible fluid flows}.
\newblock De {G}ruyter Series in Applied and Numerical Mathematics 3. De
  {G}ruyter, Berlin, 2018.

\bibitem{BrFeHo2018C}
D.~Breit, E.~Feireisl, and M.~Hofmanov{\'a}.
\newblock Markov selection for the stochastic compressible {N}avier--{S}tokes
  system.
\newblock {\em Ann. Appl. Probab.}, {\bf 30}(6):2547--2572, 2020.

\bibitem{BrFeHo2017}
D.~Breit, E.~Feireisl, M.~Hofmanov\'{a}, and B.~Maslowski.
\newblock Stationary solutions to the compressible {N}avier-{S}tokes system
  driven by stochastic forces.
\newblock {\em Probab. Theory Related Fields}, {\bf 174}(3-4):981--1032, 2019.

\bibitem{BrHo}
D.~Breit and M.~Hofmanov{\'a}.
\newblock Stochastic {N}avier-{S}tokes equations for compressible fluids.
\newblock {\em Indiana Univ. Math. J.}, {\bf 65}:1183--1250, 2016.

\bibitem{BF}
Dominic Breit and Eduard Feireisl.
\newblock Stochastic {N}avier-{S}tokes-{F}ourier equations.
\newblock {\em Indiana Univ. Math. J.}, 69(3):911--975, 2020.

\bibitem{CaSt}
N.F. Carnahan and K.E. Starling.
\newblock Equation of state for nonattracting rigid spheres.
\newblock {\em J. Chem. Phys.}, {\bf 51}:635--638, 1980.

\bibitem{DPZ96}
G.~Da~Prato and J.~Zabczyk.
\newblock {\em Ergodicity for infinite-dimensional systems}, volume 229 of {\em
  London Mathematical Society Lecture Note Series}.
\newblock Cambridge University Press, Cambridge, 1996.

\bibitem{DaPDeb}
G.~Da{P}rato and A.~Debussche.
\newblock Ergodicity for the 3{D} stochastic {N}avier--{S}tokes equations.
\newblock {\em J. Math. Pures Appl.}, {\bf 82}:877--947, 2003.

\bibitem{DL}
R.J. DiPerna and P.-L. Lions.
\newblock Ordinary differential equations, transport theory and {S}obolev
  spaces.
\newblock {\em Invent. Math.}, {\bf 98}:511--547, 1989.

\bibitem{FanFeiHof}
F.~Fanelli, E.~Feireisl, and M.~Hofmanov{\'a}.
\newblock Ergodic theory for energetically open compressible fluid flows.
\newblock {\em {\bf arxiv preprint No. 2006.02278}}, 2020.

\bibitem{EF53}
E.~Feireisl.
\newblock Propagation of oscillations, complete trajectories and attractors for
  compressible flows.
\newblock {\em NoDEA}, {\bf 10}:33--55, 2003.

\bibitem{FP9}
E.~Feireisl and H.~Petzeltov{\'a}.
\newblock Large-time behaviour of solutions to the {N}avier-{S}tokes equations
  of compressible flow.
\newblock {\em Arch. Rational Mech. Anal.}, {\bf 150}:77--96, 1999.

\bibitem{FP16}
E.~Feireisl and H.~Petzeltov{\'a}.
\newblock Zero-velocity-limit solutions to the {N}avier-{S}tokes equations of
  compressible fluid revisited.
\newblock {\em Ann. Univ. Ferrara}, 46:209--218, 2000.

\bibitem{FP15}
E.~Feireisl and H.~Petzeltov{\'a}.
\newblock Asymptotic compactness of global trajectories generated by the
  {N}avier-{S}tokes equations of compressible fluid.
\newblock {\em J. Differential Equations}, {\bf 173}:390--409, 2001.

\bibitem{FeiPr}
E.~Feireisl and D.~Pra{\v z}{\'a}k.
\newblock {\em Asymptotic behavior of dynamical systems in fluid mechanics}.
\newblock AIMS, Springfield, 2010.

\bibitem{FZ}
Eduard Feireisl and Ping Zhang.
\newblock Quasi-neutral limit for a model of viscous plasma.
\newblock {\em Arch. Ration. Mech. Anal.}, 197(1):271--295, 2010.

\bibitem{FlaGat}
F.~Flandoli and D.~Gatarek.
\newblock Martingale and stationary solutions for stochastic {N}avier-{S}tokes
  equations.
\newblock {\em Probab. Theory Related Fields}, {\bf 102}(3):367--391, 1995.

\bibitem{FlaRom}
F.~Flandoli and M.~Romito.
\newblock Markov selections for the 3{D} stochastic {N}avier-{S}tokes
  equations.
\newblock {\em Probab. Theory Related Fields}, {\bf 140}(3-4):407--458, 2008.

\bibitem{FMRT}
C.~Foias, O.~Manley, R.~Rosa, and R.~Temam.
\newblock {\em Navier-{S}tokes equations and turbulence}, volume~83 of {\em
  Encyclopedia of Mathematics and its Applications}.
\newblock Cambridge University Press, Cambridge, 2001.

\bibitem{GALN}
G.~P. Galdi.
\newblock {\em An introduction to the mathematical theory of the {N}avier -
  {S}tokes equations, Second Edition}.
\newblock Springer-Verlag, New York, 2003.

\bibitem{GEHEHI}
M.~Gei{\ss}ert, H.~Heck, and M.~Hieber.
\newblock On the equation {${\rm div}\,u=g$} and {B}ogovski\u\i's operator in
  {S}obolev spaces of negative order.
\newblock In {\em Partial differential equations and functional analysis},
  volume 168 of {\em Oper. Theory Adv. Appl.}, pages 113--121. Birkh\"auser,
  Basel, 2006.

\bibitem{HaiMat}
M.~Hairer and J.~C. Mattingly.
\newblock Ergodicity of the 2{D} {N}avier-{S}tokes equations with degenerate
  stochastic forcing.
\newblock {\em Ann. of Math. (2)}, {\bf 164}(3):993--1032, 2006.

\bibitem{ItNo}
K.~It\^o and M.~Nisio.
\newblock On stationary solutions of a stochastic differential equation.
\newblock {\em J. Math. Kyoto Univ.}, {\bf 4}:1--75, 1964.

\bibitem{KLM04AEHS}
J.~Kolafa, S.~Labik, and A.~Malijevsky.
\newblock Accurate equation of state of the hard sphere fluid in stable and
  mestable regions.
\newblock {\em Phys. Chem. Chem. Phys.}, {\bf 6}:2335--2340, 2004.

\bibitem{Li2}
P.-L. Lions.
\newblock Existence globale de solutions pour les {\'e}quations de {N}avier-
  {S}tokes compressible isentropiques.
\newblock {\em C.R. Acad. Sci. Paris, S{\'e}r I.}, {\bf 316}:1335--1340, 1993.

\bibitem{LI4}
P.-L. Lions.
\newblock {\em Mathematical topics in fluid dynamics, Vol.2, Compressible
  models}.
\newblock Oxford Science Publication, Oxford, 1998.

\bibitem{NOS1}
A.~Novotn{\'y} and I.~Stra{\v s}kraba.
\newblock Stabilization of weak solutions to compressible {N}avier-{S}tokes
  equations.
\newblock {\em J. Math. Kyoto Univ.}, {\bf 40}:217--245, 2000.

\bibitem{NOST}
A.~Novotn{\'y} and I.~Stra{\v s}kraba.
\newblock Convergence to equilibria for compressible {N}avier-{S}tokes
  equations with large data.
\newblock {\em Annali Mat. Pura Appl.}, {\bf 169}:263--287, 2001.

\bibitem{MR2386571}
M.~Romito.
\newblock Analysis of equilibrium states of {M}arkov solutions to the 3{D}
  {N}avier-{S}tokes equations driven by additive noise.
\newblock {\em J. Stat. Phys.}, 131(3):415--444, 2008.

\bibitem{Romi}
M.~Romito.
\newblock Existence of martingale and stationary suitable weak solutions for a
  stochastic {N}avier-{S}tokes system.
\newblock {\em Stochastics}, {\bf 82}(1-3):327--337, 2010.

\bibitem{SEL}
G.~R. Sell.
\newblock Global attractors for the three-dimensional {N}avier-{S}tokes
  equations.
\newblock {\em J. Dynamics Differential Equations}, {\bf 8}(1):1--33, 1996.

\bibitem{Smith}
Scott~A. Smith.
\newblock Random perturbations of viscous, compressible fluids: global
  existence of weak solutions.
\newblock {\em SIAM J. Math. Anal.}, 49(6):4521--4578, 2017.

\bibitem{VisFur}
M.~J. Vishik and A.~V. Fursikov.
\newblock {\em Mathematical problems of statistical hydromechanics}, volume~9
  of {\em Mathematics and its Applications (Soviet Series)}.
\newblock Kluwer Academic Publishers Group, Dordrecht, 1988.
\newblock Translated from the 1980 Russian original [ MR0591678] by D. A.
  Leites.

\end{thebibliography}
\bibliographystyle{plain}

\end{document}